\documentclass[12pt,leqno]{amsart}
\usepackage[
margin=1.5in,
]{geometry}
\usepackage{tikz,tikz-cd}
\usepackage{url}
\usepackage{graphicx}
\usepackage{graphicx}
\usepackage{color}
\usepackage{times}
\usepackage{cite}
\usepackage{enumerate,latexsym}
\usepackage{latexsym}
\usepackage{amsmath,amssymb}
\usepackage{graphicx}
\usepackage{amsthm}
\usepackage{verbatim}
\newtheorem{thm}{Theorem}[section]
\newtheorem*{theorem*}{Theorem}
\newtheorem*{acknowledgement*}{Acknowledgement}
\newtheorem{cor}[thm]{Corollary}

\newtheorem{lem}[thm]{Lemma}
\newtheorem{prop}[thm]{Proposition}
\theoremstyle{definition}

\theoremstyle{remark}
\newtheorem{rem}[thm]{Remark}

\theoremstyle{definition}
\newtheorem{definition}[thm]{Definition}
\numberwithin{equation}{section}

\newcommand{\set}[1]{\left\{#1\right\}}
\newcommand{\Real}{\mathbb R}

\newcommand{\dist}[0]{\mathrm{dist}}

\newcommand{\ct}[0]{\mathrm{ct}}
\title{Colding-Minicozzi Entropies in Cartan-Hadamard Manifolds}

	\author{Jacob Bernstein}
\address{Department of Mathematics, Johns Hopkins University, 3400 N. Charles Street, Baltimore, MD 21218}
\email{bernstein@math.jhu.edu}

\author{Arunima Bhattacharya}
\address{Mathematical Sciences Research Institute/ Simons Laufer Mathematical Sciences Institute, 17 Gauss Way,
Berkeley, CA 94720 }
\email{arunimab@msri.org}

\begin{document}

\begin{abstract}
We introduce a family of functionals on submanifolds of Cartan-Hadamard manifolds that generalize the Colding-Minicozzi entropy of submanifolds of Euclidean space. We show that these functionals are monotone under mean curvature flow under natural conditions. 
As a consequence, we obtain sharp lower bounds on these entropies for certain closed hypersurfaces and observe a novel rigidity phenomenon.
\end{abstract}
\maketitle
\section{Introduction}
In \cite{Coldinga}, Colding and Minicozzi introduced the following functional on the space of $n$-dimensional submanifolds $\Sigma\subset \Real^{n+k}$,
$$
\lambda[\Sigma]=\sup_{\mathbf{x}_0\in \Real^{n+k}, \tau>0} (4\pi \tau)^{-\frac{n}{2}}\int_{\Sigma} {e^{-\frac{|\mathbf{x}(x)-\mathbf{x}_0|^2}{4\tau}}} dVol_{\Sigma}(x).
$$
Here, $(t,x)\mapsto (4\pi t)^{-\frac{n}{2}} {e^{-\frac{|\mathbf{x}(x)|^2}{4t}}}$ is the heat kernel of $\Real^n$.  In the definition of $\lambda[\Sigma]$, this kernel is extended to all of $\Real^{n+k}$ in the obvious way.
Colding and Minicozzi called this quantity the \emph{entropy} of $\Sigma$ and observed that, by Huisken's monotonicity formula \cite{HuiskenMon}, it is monotone along reasonable mean curvature flows.  That is, if $t\in [0,T)\mapsto \Sigma_t\subset \Real^{n+k}$ is a family of complete $n$-dimensional submanifolds of polynomial volume growth that satisfy
$$
\frac{d\mathbf{x}}{dt}=\mathbf{H}_{\Sigma_t},
$$
then $\lambda[\Sigma_t]$ is monotone non-increasing in $t$.  Here $\mathbf{H}_{\Sigma_t}=\Delta_{\Sigma_t} \mathbf{x}=-H_{\Sigma_t} \mathbf{n}_{\Sigma_t}$ is the mean curvature vector of $\Sigma_t$ and a submanifold, $\Sigma$, has polynomial volume growth if $Vol_{\Real}(\Sigma\cap B_R)\leq N (1+R)^N$ for some $N>0$.  This notion of entropy has been extensively studied in recent years.  See, for instance, \cite{BWIsotopy, BernsteinWang1, JZhu, CIMW, BernsteinWang2, BernsteinWang3, BWHausdorff, SWang, KetoverZhou}.  

 In \cite{BernsteinHypEntropy}, an analogous quantity was introduced by the first author for submanifolds of hyperbolic space, $\mathbb{H}^{n+k}$.  In particular, it was shown that there is a notion of entropy that is monotone non-increasing along mean curvature flows that have exponential volume growth, i.e.,
$$
Vol_{\mathbb{H}}(\Sigma\cap \mathcal{B}^{\mathbb{H}}_R(x_0)) \leq N e^{NR}
$$
for some $N>0$.    See \cite{mramorEntropyGenericMean2018, zhuGeometricVariationalProblems2018, sunEntropyClosedManifold2019, sunEntropyClosedManifoldCorrection} for other extensions of the concepts to curved ambient manifolds.  For applications of the Colding-Minicozzi entropy in hyperbolic space, we refer to \cite{bernsteinRAPAMS} and \cite{JunfuYaoMountainPass}.

In this note, we generalize and unify both quantities by introducing a one-parameter family of entropies defined on submanifolds of any Cartan-Hadamard manifold.    Indeed, suppose that $(M,g)$ is a Cartan-Hadamard manifold, i.e., a complete simply connected Riemannian manifold with non-positive sectional curvature.  For any $n\geq 1$ and $\kappa\geq 0$, let
\begin{equation}\label{KnkappaEqn}
K_{n,\kappa}(t,r)= \left\{\begin{array}{cc} \kappa^nK_n(\kappa^2 t, \kappa r) & \kappa,t>0,  r\geq 0\\
	(4\pi t)^{-n/2} e^{-\frac{r^2}{4 t}} & \kappa=0, t>0, r\geq  0
\end{array}\right.
\end{equation}
where $K_n$ is the function introduced in \cite{daviesHeatKernelBounds1988} and used in \cite{BernsteinHypEntropy} -- see Section \ref{HeatKernelSec}.  In particular, this function makes
$$
H_{n, \kappa}(t,x; t_0, x_0)=K_{n,\kappa}(t-t_0, \dist_{g}(x,x_0))
$$
the heat kernel on $(M,g)$ with singularity at $x_0$ and time $t_0$ precisely when $(M,g)$ is the space form of constant curvature $-\kappa^2$.

Following \cite{Coldinga} and \cite{BernsteinHypEntropy}, let
$$
\Phi_{n, \kappa}^{t_0,x_0}(t,x)=K_{n,\kappa}(t_0-t, \dist_{g}(x,x_0))
$$
and for $\Sigma\subset M$, an $n$-dimensional submanifold,  define \emph{the Colding-Minicozzi $\kappa$-entropy of $\Sigma$ in $(M,g)$} to be
$$
\lambda_g^\kappa[\Sigma]=\sup_{x_0\in M, \tau>0}\int_{\Sigma} \Phi_{n, \kappa}^{0, x_0} (-\tau, x) dVol_{\Sigma}(x)=\sup_{x_0\in M, \tau>0}\int_{\Sigma} \Phi_{n, \kappa}^{\tau, x_0} (0, x) dVol_{\Sigma}(x).
$$
When $\kappa=0$ and $(M,g)=(\Real^{n+k}, g_{\Real})$ is Euclidean space, this is just the usual Colding-Minicozzi entropy, $\lambda[\Sigma]$, of $\Sigma$. When $\kappa=1$ and $(M,g)=(\mathbb{H}^{n+k}, g_{\mathbb{H}})$ is hyperbolic space, this is the entropy in hyperbolic space, $\lambda_{\mathbb{H}}[\Sigma]$, introduced in \cite{BernsteinHypEntropy}.

The main result of this paper is a monotonicity property of an appropriate range of entropies along well-behaved mean curvature flows when the ambient Cartan-Hadamard manifold has sectional curvature bounded above by $-\kappa^2_0$.
\begin{thm}\label{MainMonThm}
	Let $(M,g)$ be an $(n+k)$-dimensional Cartan-Hadamard manifold with $\sec_{g}\leq -{\kappa}^2_0$. If $0\leq \kappa \leq \kappa_0$ and $t\in [t_1,t_2]\mapsto \Sigma_{t}\subset M$ is a locally smooth mean curvature flow of $n$-dimensional complete submanifolds of exponential volume growth and $Ric_g \geq -\Lambda^2 g$ on $\Sigma_t$, for $t\in [t_1,t_2]$ and a uniform $\Lambda\geq 0$, then  
$$
\lambda_{g}^\kappa[\Sigma_{t_1}]\geq \lambda_{g}^\kappa[\Sigma_{t_2}].
$$

In particular, in any Cartan-Hadamard manifold,  $t\mapsto \lambda^0_g[\Sigma_t]$ is monotone non-increasing along a classical mean 
curvature flow of closed hypersurfaces.
\end{thm}
As observed in \cite{BernsteinHypEntropy}, this monotonicity property allows one to generalize the results of \cite{BWIsotopy, BernsteinWang1, JZhu, BernsteinWang2, BernsteinWang3, BWHausdorff, SWang} on properties of the Colding-Minicozzi entropy of closed hypersurfaces in Euclidean space to the curved setting.  The extension to spaces of variable curvature also leads to an interesting rigidity phenomenon that is reminiscent of the classical Cartan-Hadamard conjecture on the sharp isoperimetric inequality in Cartan-Hadamard manifolds. 
\begin{thm}\label{MainRigidThm}
	Let $(M,g)$ be an $(n+1)$-dimensional Cartan-Hadamard manifold.  If $\Omega\subset M$ is a compact domain with $\partial \Omega$ smooth and mean convex, then
	$$
	\lambda_g^0[\partial \Omega]\geq  \lambda[\mathbb{S}^n].
	$$
	Equality holds if and only if, for some $R>0$, $(\Omega, g|_{\Omega})$ is isometric to $(\bar{B}_R, g_{\Real}|_{\bar{B}_R})$, the standard Euclidean metric restricted to the closed Euclidean ball in $\Real^{n+1}$ of radius $R$.
	When $n=1$ or $2$, the same result is true without the hypothesis that $\partial \Omega$ is mean convex.
\end{thm}
\begin{rem}
When the ambient manifold  is flat and $n=1$ or $2$ this is the main result of \cite{Grayson1989} or \cite{BernsteinWang2}. When $n\geq 3$ this is a special, and easier, case of \cite{BernsteinWang1} and \cite{JZhu}. Without the hypothesis of mean convexity, the arguments of \cite{BernsteinWang1} and \cite{JZhu} require a suitable weak formulation of mean curvature flow in order to flow past singularities.  A subtlety that complicates things in the curved setting is that without a uniform lower bound on the Ricci curvature certain pathological behavior may arise -- see \cite{Ilmanen1992,Ilmanen1993} and Remark \ref{RicCondRem}.  This seems to be only a technical issue, but for the sake of simplicity, we impose hypotheses that allow us to remain in the setting of classical mean curvature flows.  
\end{rem}

As a final remark, we mention that in \cite{BernsteinHypEntropy} the first author showed that if $\Sigma$ is an $n$-dimensional submanifold of $\mathbb{H}^{n+k}$ that is regular up to the ideal boundary, then there is an inequality  between $\partial_\infty\Sigma$, the conformal volume of the ideal boundary of $\Sigma$, which is a subset of the ideal boundary of $\mathbb{S}^{n+k-1}=\partial_\infty \mathbb{H}^{n+k}$ the ideal boundary of hyperbolic space, and $\lambda_{\mathbb{H}}[\Sigma]$, the Colding-Minicozzi entropy of $\Sigma$ in hyperbolic space.  In forthcoming work  \cite{BernsteinBhattacharyaCH}, we will use the monotonicity properties of the present paper to study an analogous relationship for certain minimal submanifolds of complex hyperbolic space.

\subsection*{Acknowledgements}
The authors thank Letian Chen and Junfu Yao for their helpful comments and suggestions. The first author was partially supported by the NSF Grant DMS-1904674 and  DMS-2203132 and the Institute for Advanced Study with funding provided by the Charles Simonyi Endowment.
The second author was partially supported by the AMS-Simons Travel Grant and funding from the Simons Laufer Mathematical Sciences Institute.

\section{Notation and background}
Unless stated otherwise, throughout this paper we suppose that $(M,g)$ is a Cartan-Hadamard manifold, i.e., a complete simply connected manifold with non-positive sectional curvature.  We denote by $(\Real^{l}, g_{\Real})$ the Euclidean space and by $(\mathbb{H}^l(\kappa), g_{\mathbb{H}(\kappa)})$ the space form of constant curvature $-\kappa^2$.  When $\kappa=1$, i.e., for hyperbolic space we suppress $\kappa$ from the notation.  These are all Cartan-Hadamard manifolds.

 We will repeatedly use that there is a unique minimizing geodesic connecting any two points in a Cartan-Hadamard manifold and that the distance to any $x_0\in M$ is a smooth function on $M\setminus\set{x_0}$ and the square of this function is smooth on $M$. For any fixed point, $x_0\in M$, $\rho:M \to \Real$ denotes the distance function with respect to $g$ between the points $x$ and $x_0$
  $$\rho(x)=\rho(x; x_0)=\dist_g(x,x_0).$$
  When $x_0$ is clear from the context we write $\rho(x)$.   For a $C^1$ function $F:M\to \Real$ we write
$$
\partial_\rho F=g(\nabla_g \rho, \nabla_g F)=d\rho(\nabla_g F)=dF(\nabla_g \rho)
$$
which is continuous on $M\setminus\set{x_0}$.  If $\nabla_g F|_{x=x_0}=0$, then this extends continuously to $M$.  When $F$ is $C^2$ we also write
$$
\partial^2_\rho F= \nabla_g^2 F(\nabla_g \rho, \nabla_g \rho).
$$

By $\mathcal{B}^g_{R}(x)$, we denote a ball of radius $R$ around $x$ with respect to the metric $g$ in the manifold $M$. By $B_{R}(x)$, we denote a ball of radius $R$ around $x$ in the Euclidean space.  A submanifold $\Sigma\subset M$ has exponential volume growth if there is a constant $N>0$ so, for some $x_0\in M$,
$$
Vol_g(\mathcal{B}_R^g(x_0)\cap \Sigma)\leq N e^{NR}.
$$ 

Given vectors $\mathbf{v}, \mathbf{w}\in T_x M$ we will sometimes 
write
$$
\mathbf{v}\cdot \mathbf{w}=g_x(\mathbf{v}, \mathbf{w})
$$
and
$$
|\mathbf{v}|^2=\mathbf{v}\cdot \mathbf{v}=g_x(\mathbf{v}, \mathbf{v}).
$$
When $\mathbf{v}$ and $\mathbf{w}$ are linearly independent we denote the plane they span by
$$
\mathbf{v}\wedge \mathbf{w}=\mathrm{span}\set{\mathbf{v},\mathbf{w}}\in G_2(T_xM).
$$
We will suppress the point $x$ in the notation when it is clear from the context.

\subsection{Mean curvature flows}
Let $(M,g)$ be a general $(n+k)$-dimensional Riemannian manifold and $I\subset \Real$ an interval with non-empty interior. Consider a family $t\in I \mapsto \Sigma_t\subset M$ of smooth properly embedded $n$-dimensional submanifolds.
We say this family is a \emph{(classical) mean curvature flow (in $M$)} if there is a fixed $n$-dimensional manifold $\Sigma$ and a smooth family of maps
$$
F:I \times \Sigma\to M
$$
satisfying, $F(t, \cdot):\Sigma \to \Sigma_t\subset M$ is a proper embedding for each $t\in I$ and
\begin{equation}\label{ClassicalMCFEqn}
\left(\frac{\partial F}{\partial t} (t,x)\right)^\perp=\mathbf{H}_{\Sigma_t}(F(t,x)).
\end{equation}
Here $\left(\frac{\partial F}{\partial t} (t,x)\right)^\perp\in T_{F(t,x)}M$ is the normal component of $\frac{\partial F}{\partial t}(t,x)$.  Observe that if $\Sigma$ is closed and $I$ is compact, then
$$
F(I\times \Sigma)=\bigcup_{t\in I}\Sigma_t
$$
is a compact subset of $M$.  We remark that \eqref{ClassicalMCFEqn} is invariant under (time varying) reparameterizations.  In many cases of interest, one can apply a reparameterization to obtain a solution to
$$
\frac{\partial F}{\partial t}(t,x)=\mathbf{H}_{\Sigma_t}(F(t,x)).
$$
That is, the embeddings move normally to the hypersurfaces with velocity given by the mean curvature vector. 

We say the family $t\in I \mapsto \Sigma_t\subset M$ is a \emph{locally smooth mean curvature flow} if, for every $t_0\in I$ and $x_0\in M$, there is an $\epsilon>0$ so if $I'=I\cap \set{t: |t-t_0|<\epsilon^2}$, then  $t\in I'\mapsto \mathcal{B}_\epsilon^g(x_0)\cap \Sigma_t$ is a, possibly empty, classical mean curvature flow in $\mathcal{B}_\epsilon^g(x_0)$. 
Every classical mean curvature flow is automatically a locally smooth flow, but locally smooth flows need not be classical -- for instance the flow of Remark \ref{RicCondRem}.

\section{Fundamental solutions of the heat equation and the heat kernel}\label{HeatKernelSec}
We recall some basic properties of the fundamental solutions of the heat equation on a Riemannian manifold.  Our source is \cite{Dodziuk} and \cite{daviesHeatKernelBounds1988} -- we refer the reader also to \cite{PeterLiBook}.  Note we follow the convention of \cite{daviesHeatKernelBounds1988} and put the time variable first.
\subsection{Background and terminology}
Let $(M,g)$ be a connected Riemannian manifold, not necessarily complete.  Following Dodziuk \cite{Dodziuk},  a continuous function $p: (0,\infty)\times M\times M\to \Real $ is a \emph{fundamental solution of the heat equation} if, for every bounded continuous function $u_0$ on $M$, the function
$$
u(t,x)=\left\{\begin{array}{cc} \int_{M} p(t,x,y) u_0(y) dVol_{g}(y) &  t>0 \\ u_0(x) & t=0  \end{array} \right.
$$
is a solution of the Cauchy problem with initial data $u_0$.  That is, $u\in C^0( [0, \infty)\times M)\cap C^{1}((0, \infty)\times M)$ and for all $t>0$, $u\in C^2(M)$ and
$$
\left\{
\begin{array}{cc} \frac{du}{dt} -\Delta_{g} u=0 & \mbox{on } (0, \infty)\times M\\ 
   u(0,x)=u_0(x) & \mbox{on } M.
  \end{array}
  \right.
$$
By \cite[Theorem 2.2]{Dodziuk}, if $(M,g)$ is complete and has a uniform lower bound on its Ricci curvature, then $p$ is unique, but this need not be the case in general.  However, as observed in \cite{Dodziuk} one may always choose a unique fundamental solution by selecting the one that is minimal.   This selection, $p_M$, which we call the \emph{minimal fundamental solution of the heat equation} is characterized by
$$
p_M(t,x,y)=\sup_{D\in \mathcal{D}(M)} p_D(t,x,y)
$$
where $\mathcal{D}(M)$ is the collection of precompact connected open sets $D\subset M$ with $\partial D$ smooth and $p_D$ is the, necessarily unique, fundamental solution of the heat equation in $D$ with Dirichlet boundary conditions.  We refer to \cite[Lemma 3.2 and Theorem 3.6]{Dodziuk} for a list of properties of $p_D$ and $p_M$.

The minimal fundamental solution satisfies, for $t>0$,
$$
\int_{M} p_M(t,x,y) dVol_{g}(y) \leq 1.
$$
In general, one cannot expect equality even when $M$ is complete.
However, when $(M,g)$ is complete and has a uniform lower bound on the Ricci curvature Dodziuk shows in \cite[Theorem 4.2]{Dodziuk} that, in addition to $p_M$ being the unique fundamental solution to the heat equation, one also has conservation of total heat for all $t>0$
$$
\int_{M} p_M(t,x,y) dVol_g(y) =1, x\in M.
$$
In this case, we refer to $p_M$ as the \emph{heat kernel} of $(M,g)$  -- see \cite{Pigola,Azencott} for other conditions that ensure this property holds.

\subsection{Properties of the heat kernel of $\mathbb{H}^n$}\label{Hypspace}
We specialize now to hyperbolic space, $\mathbb{H}^n$ -- the space form of constant sectional curvature $-1$.  As $(\mathbb{H}^n, g_{\mathbb{H}})$ is complete and has constant Ricci curvature there is a unique heat kernel $p_{\mathbb{H}^n}$ in the sense above.  Following the notation of \cite{BernsteinHypEntropy} and \cite{daviesHeatKernelBounds1988}, let us write 
$$
H_n(t,x; t_0, x_0)=p_{\mathbb{H}^n}(t-t_0, x, x_0)
$$
which we think of as a function
$$
H_n(\cdot, \cdot; t_0, x_0): (t_0, \infty)\times \mathbb{H}^n\to \Real
$$
and call \emph{the heat kernel on $\mathbb{H}^n$ with singularity at $x=x_0$ and $t=t_0$}. 
Formally, this is a solution to 
$$
\left\{\begin{array}{cc} \left( \frac{\partial}{\partial t}-\Delta_{\mathbb{H}}\right)H_n=0 & t>t_0\\
	\lim_{t\downarrow t_0} H_n =\delta_{x_0}\end{array}\right.
$$
where $\delta_{x_0}$ is the Dirac delta at $x_0$.

The symmetries of $\mathbb{H}^n$ and uniqueness of the heat kernel ensure that if $H_n(t,x; t_0, x_0)$ is the heat kernel on $\mathbb{H}^n$ with singularity at $x=x_0$ and $t=t_0$, then, there is a smooth positive function $K_n:(0, \infty)\times [0, \infty)\to \Real$ such that
$$
H_n(t,x; t_0, x_0)=K_n(t-t_0, \dist_{\mathbb{H}}(x,x_0)).
$$
For instance, as $\mathbb{H}^1$ is just the Euclidean line,
$$
K_1(t,r)=(4\pi t)^{-1/2}e^{-\frac{r^2}{4t}}.
$$
Likewise,
$$
K_3(t,r)=(4\pi t)^{-3/2}\frac{r}{\sinh(r)}e^{-t-\frac{r^2}{4t}}.
$$
Note that for $t>t_0$ one has
$$
\nabla_g H_n(t,x;t_0,x_0)|_{x=x_0}=0.
$$

Properties of $K_n$ were extensively studied by Davies and Mandouvalos in \cite{daviesHeatKernelBounds1988} and we will make heavy use of their results, which, for the convenience of readers, we summarize below. We first recall two recurrence relations for $K_n$ from \cite{daviesHeatKernelBounds1988}.
The first is attributed to Millison:
\begin{equation}\label{MillisonEqn}
	K_{n+2}(t,r) =-\frac{e^{-nt}}{2\pi \sinh(r)} \partial_r K_n(t,r).
\end{equation}
The second is:
\begin{equation} \label{IntegralEqn}
	K_{n,0}(t,r)=\int_{r}^\infty \frac{e^{\frac{1}{4}\left(2n-1\right)t}K_{n+1}(t,s) \sinh(s)}{\left(\cosh(s)-\cosh(r)\right)^{\frac{1}{2}}} ds.
\end{equation}

Using these relations, Davies and Mandouvalos obtained the following uniform estimate on $K_{n+1}$ \cite[Thereom 3.1]{daviesHeatKernelBounds1988}:
\begin{equation}\label{KnDecayEstUB}
	K_{n+1}(t,r)\leq C_n t^{-\frac{1}{2}(n+1)}e^{-\frac{1}{4}n^2 t-\frac{r^2}{4t}-\frac{1}{2}nr} (1+r+t)^{\frac{1}{2}n -1} (1+r)
\end{equation}
and a corresponding lower bound
\begin{equation}\label{KnDecayEstLB}
	 C_n^{-1} t^{-\frac{1}{2}(n+1)}e^{-\frac{1}{4}n^2 t-\frac{r^2}{4t}-\frac{1}{2}nr} (1+r+t)^{\frac{1}{2}n -1} (1+r)\leq K_{n+1}(t,r)
\end{equation}
where $C_n>1$ depends only on $n$.
In particular, for any fixed $x_0\in \mathbb{H}^{n+1}$ and $R>0$, one has for $t\geq 1$
\begin{equation}\label{TimeDecayEqn}
	\sup_{x\in \mathcal B_{R}^{\mathbb{H}}(x_0)} H_{n+1}(t,x; 0,x_0)\leq C_{n,R}' t^{-\frac{3}{2}} e^{-\frac{1}{4}n^2 t}.
\end{equation}

Next, we recall the following super-convexity estimate of $K_n$ shown for $n\leq 4$ in \cite{BernsteinHypEntropy} and for all $n$ in \cite{zhangSuperconvexityHeatKernel2021}. This is equivalent to a super-convexity property of $H_n$
\begin{prop}\label{KnProp}
	For all $n\geq 1$, 
	\begin{equation}\label{KconvexityEqn}
		\partial^2_r \log K_n -\coth(r) \partial_r \log K_n \geq 0.
	\end{equation} 
Moreover, this inequality is strict for $r>0$.
\end{prop}
%
%
%

\section{A characterization of constant curvature Cartan-Hadamard manifolds}

%

We use the functions $K_{n,\kappa}$ of \eqref{KnkappaEqn}, which are rescalings of the functions $K_n$ introduced in Section \ref{Hypspace}, in order to characterize constant curvature Cartan-Hadamard manifolds in terms of the minimal fundamental solution of the heat equation.  Observe no \emph{a priori} assumptions are made to ensure that this is the unique fundamental solution -- i.e., the heat kernel in our terminology.
\begin{thm}\label{ConstCurvCharThm}
	Fix $\kappa\geq 0$ and let $(M,g)$ be an $n$-dimensional Cartan-Hadamard manifold with $n\geq 2$ and $\mathrm{sec}_{g} \leq -\kappa^2$ and denote by $p_M$ the minimal fundamental solution of the heat equation on $M$.  For all $t>0$, one has
	$$
	0<p_M(t,x,y)\leq K_{n, \kappa} (t,d_g(x,y)).
	$$
	This inequality is strict unless $(M,g)$ has constant sectional curvature $-\kappa^2$, in which case $p_M(t,x,y)=K_{n,\kappa}(t,d_g(x,y))$ is the heat kernel of $(M,g)$.  
\end{thm}

Before proving this result we will need to first prove some auxiliary results. 
In what follows, it is convenient to introduce the following notation:
$$
\mathrm{ct}_{\kappa}(r)=\left\{\begin{array}{cc} \kappa \coth(\kappa r) & \kappa, r>0\\
																	\frac{1}{r} & \kappa=0, r>0. 
																	\end{array}\right.
$$
We next recall the following sharp statement of the Hessian comparison theorem:
\begin{prop}\label{HessCompProp}
Let $(M,g)$ be a Cartan-Hadamard manifold that satisfies $\mathrm{sec}_g\leq -\kappa^2 \leq 0$.
Fixing $x_0\in M$, for any $x\neq x_0$ and $v\in T_x M$ with $v\neq 0$  one has:
$$
(\nabla^2_g \rho)_x(v,v) \geq \ct_\kappa(\rho(x)) (g_x(v,v)-(d\rho(v))^2)
$$
and this inequality is strict unless $v$ is parallel to $\nabla_g \rho(x)$ or
$$
\mathrm{sec}_g(\gamma'(t)\wedge V(t))=-\kappa^2 
$$
for all $t\in [0, \rho(x)]$, 
where $\gamma:[0, \rho(x)]\to M$ is the unique unit speed geodesic connecting $x_0$ to $x$ and $V:[0, \rho(x)]\to TM$ is the parallel transport of $v$ along $\gamma$.

A consequence of this is a reverse form of Laplacian comparison
$$
(\Delta_g \rho)(x)\geq (n-1) \ct_{\kappa} (\rho(x)),
	$$
which is a strict inequality unless, 
$$\mathrm{sec}_g(\gamma'(t)\wedge W)=-\kappa^2 
$$
for all $t\in [0, \rho(x)]$ and all $W\in T_{\gamma(t)}M$ that are orthogonal to $\gamma'(t)$.
\end{prop}
\begin{proof}
As $(M,g)$ is a Cartan-Hadamard manifold there is a unique length minimizing geodesic connecting $x_0$ to $x$ and so the inequality is the statement of the Hessian comparison theorem -- e.g., \cite[Theorem 1.1]{Schoen1994}.  The strictness of the inequality follows from the proof -- namely because equality forces $V$ to be proportional to a Jacobi field.
\end{proof}


The definition of $K_{n,\kappa}$ in \eqref{KnkappaEqn} and properties of $K_n$ ensure that for all $t>0$, 
$$
\partial_t K_{n,\kappa} (t,r)=\partial_{r}^2 K_{n,\kappa} (t, r)+(n-1) \mathrm{ct}_{\kappa}(r) \partial_r K_{n,\kappa}(t, r).
$$
Observe that for $r,t>0$, the Millison identity \eqref{MillisonEqn}  and the fact that $K_{n, \kappa}$ is positive for all $t>0$ imply
$$
\partial_r K_{n, \kappa} (t,r)= -2\pi e^{n \kappa^2 t} \kappa^{-1} \sinh (\kappa r) K_{n+2, \kappa}(t,r)<0,
$$
and so, for $\kappa, r, t>0$, 
$$
\partial_r K_{n,\kappa}(t,r) <0.
$$
When $\kappa=0$ and $r,t>0$,  one may directly compute that
$$
\partial _{r}K_{n,0} (t, r)= -\frac{r}{2t} K_{n,0}(t,r)=-2\pi \rho K_{n+2, 0}(t,r)<0.
$$

The $K_{n, \kappa}$ can be used to give the heat kernel on the Cartan-Hadamard space forms of constant curvature $-\kappa^2$.  More generally, they give supersolutions of the heat equation on Cartan-Hadamard manifolds of curvature bounded above by $-\kappa^2$.  
\begin{lem}\label{SupersolutionLem}
	Fix $\kappa\geq 0$ and let $(M,g)$ be an $n$-dimensional Cartan-Hadamard manifold with $n\geq 2$ and  $\mathrm{sec}_{g} \leq -\kappa^2$ and fix a point $x_0\in M$.  For all $t> 0$, the function
	$$
	 H_{n,\kappa}(t,x)=K_{n,\kappa}(t,\rho(x))=K_{n, \kappa}(t, \dist_g(x,x_0))
	 $$
	 is a smooth super solution of the heat equation; i.e., 
	$$
	\left( \frac{\partial}{\partial t}-\Delta_g\right) H_{n, \kappa}\geq 0.
	$$
   Moreover, this inequality is strict at some point $x\in M$ and time $t>0$ unless $(M,g)$ has constant curvature $-\kappa^2$.
\end{lem}
\begin{proof}
	One computes at $(t,x)$ that 
	\begin{align*}
	\left(\frac{\partial}{\partial t} -\Delta_{g}\right) H_{n,\kappa} &= \left(\partial_t K_{n,\kappa}(t,r) -\partial_r^2 K_{n,\kappa} (t,r)\right)|_{r=\rho}- \partial_r K_{n,\kappa} (t,r)|_{r=\rho}\Delta_g \rho\\
	&= (n-1) \ct_{\kappa}(\rho)\partial_r K_{n,\kappa}(t,\rho) - \partial_r K_{n,\kappa}(t,\rho)  \Delta_g \rho.
	\end{align*}
	Hence, as $-\partial_{r} K_{n, \kappa}\geq 0$, Proposition \ref{HessCompProp} implies
	$$ 
		\left(\frac{\partial}{\partial t} -\Delta_{g}\right) H_{n,\kappa} \geq 0
	$$
	and the inequality is strict at $x\neq x_0$ unless 
	$$
    \sec_{g}(\gamma'(t)\wedge W)=-\kappa^2
    $$
    where $\gamma:[0,L]\to M$ is the unit speed geodesic connecting $x_0$ to $x$ and $W\in T_{\gamma(t)} M$ is orthogonal to $\gamma'$.
    
    That is, the inequality is strict somewhere unless $\sec_{g}(\nabla_{g} \rho \wedge W)=-\kappa^2$ for all $x\neq x_0$ and $W\in T_x M$ orthogonal to $\nabla_{g} \rho(x)$.  In this case, one has, by Proposition \ref{HessCompProp}, that, for all $x\neq x_0$,
    $$
    (\nabla_g^2 \rho)_x=\ct_{\kappa}(\rho(x)) (g_x-d\rho^2_x).
    $$
    Let
   \begin{equation}\label{fkEqn}
    f_{\kappa}(r)= \left\{
    \begin{array}{cc} \frac{1}{\kappa^2} \left(\cosh (\kappa r) -1\right)& \kappa>0 \\
	    	 \frac{1}{2} r^2 & \kappa=0.
	     \end{array} 
    \right.
\end{equation}
    Using the chain rule, one concludes that on $M\setminus \set{x_0}$
    $$
    (\nabla^2_g f_{\kappa}(\rho))_x=(1+\kappa^2 f_{\kappa}(\rho(x)))g_x.
    $$
In addition, as $f_{\kappa}(\rho)$ is readily seen to be smooth at $x_0$, the equality extends to all of $M$.  This conformal property of the Hessian means that the metric $g$ locally has a warped product structure -- see \cite[Theorem 4.3.3]{Petersen3rdEd} where the result is attributed to Brinkmann \cite{Brinkmann}.  Moreover, the specific form of the conformal term \cite[Corollary 4.3.4]{Petersen3rdEd} shows that the metric has constant sectional curvature $-\kappa^2$ in a neighborhood of $x_0$.  It is not hard to modify this argument to see that $(M,g)$ globally has constant curvature $-\kappa^2$.
\end{proof}

We need the following short time estimate for $H_{n, \kappa}$ near the singularity. We state this result quite generally as it will be used elsewhere.
\begin{lem}\label{ApproxEuclidHKLem}
For $k\geq 0$, let $(M,g)$ be an $(n+k)$-dimensional Riemannian manifold (not necessarily Cartan-Hadamard) and let $\Sigma\subset M$ be a smooth properly embedded $n$-dimensional submanifold (i.e., an open subset of $M$ when $k=0$).  Fix a point $x_0\in M$ and an $R>0$  so $U=\mathcal{B}_{R}^g(x_0)$ is geodesically convex and precompact in $M$.   Suppose that
	$$
	i:(U, g|_{U})\to (\Real^d, g_{\Real})
	$$ 
	is an isometric embedding.	For $x\in U\cap \Sigma$ define
	 $$
	 H(t,x)=K_{n,\kappa}(t,\rho(x))=K_{n,\kappa}(t, \dist_g(x,x_0)) \mbox{ and } G(t,x)= K_{n, 0}(t, |i(x)-i(x_0)|).
	 $$
	  There is a constant $C(U,i,n,\kappa)>0$, independent of $\Sigma$, so for $0<t\leq 1$ and $f\in L^{\infty}(U\cap \Sigma)$,
    \begin{align*}
		\left| \int_{U\cap \Sigma} f(x) \left(H(t,x) -G(t,x)\right)dVol_\Sigma(x)\right| \leq C 	t^{1/3}  \Theta(t) \Vert f \Vert_{L^\infty(U\cap \Sigma)} 
\end{align*}
where
$$
\Theta(t)=Vol_g(U\cap \Sigma)+\int_{U\cap \Sigma} G(t,x)\leq C'(\Sigma,n)<\infty.
$$
\end{lem}
\begin{proof} 
	We first claim that there exists $C_1=C_1(n ,\kappa, R)$ so, for $0<t\leq 1$ and $0\leq r \leq R$, 
	$$
	|K_{n, \kappa}(t,r)-K_{n,0}(t,r)|\leq  C_1 (t+r^2)K_{n,0}(t,r).
	$$
	This follows from \cite[Equations (3.2) and (3.3)]{daviesHeatKernelBounds1988} for odd $n$. The expansion for even $n$ follows from \eqref{IntegralEqn} above. Hence, for $0<t<\min\set{1, R^{3}}$ and $0\leq r < t^{1/3}$ we have, up to increasing $C_1$,
	$$
		|K_{n, \kappa}(t,r)-K_{n,0}(t,r)|\leq C_1 t^{2/3}K_{n,0}(t,r).
	$$
	The geometry of $i(U)$ ensures that there exists $\beta=\beta(i,R)>1$ so 
	$$
	0\leq \rho(x)-|i(x)-i(x_0)| =d_g(x,x_0)-|i(x)-i(x_0)| \leq \beta|i(x)-i(x_0)|^3.
	$$
	Hence, for $x\in \mathcal{B}^g_{t^{1/3}}(x_0)$, and $0<t\leq 1$
	\begin{align*}
	|K_{n,0}(t,\rho(x))-K_{n,0}(t, |i(x)-i(x_0)|)|&\leq \left(e^{\frac{1}{2}\beta t^{1/3}}-1\right) K_{n,0}(t, |i(x)-i(x_0)|)\\
	&\leq C_2 t^{1/3} K_{n,0}(t, |i(x)-i(x_0)|) 
	\end{align*}
where $C_2$ depends only on $\beta$.  That is, for $x\in \mathcal{B}^g_{t^{1/3}}(x_0)$, the two inequalities imply
$$
|H(t,x)-G(t,x)| \leq C_1  t^{2/3}K_{n,0}(t,\rho(x))+ C_2 t^{1/3} G(t,x).
$$
This in turn yields for $x\in \mathcal{B}^g_{t^{1/3}}(x_0)$, 
$$
|H(t,x)-G(t,x)|\leq (C_1 t^{2/3}+C_2 t^{1/3}+C_1C_2 t) G(t,x)\leq C_2't ^{1/3} G(t,x).
$$
In the above, we used $t\leq 1$. Hence,
\begin{equation}
	\begin{aligned}
 \Big|\int_{\mathcal{B}^g_{t^{1/3}}(x_0)\cap \Sigma} & f(x) \left(H(t,x)-G(t,x)\right) dVol_\Sigma(x)  \Big| \nonumber\\
& \leq C_2't^{1/3} \Vert f \Vert_{L^\infty(U\cap \Sigma)} \int_{\mathcal{B}^g_{t^{1/3}}(x_0)\cap \Sigma} G(t,x
 ) dVol_{ \Sigma} (x) \\
& \leq C_2' t^{1/3} \Theta(t) \Vert f \Vert_{L^\infty(U\cap \Sigma)}. 
 \end{aligned}
\end{equation}

\medskip
Next, we observe that by \eqref{KnDecayEstUB} for $ 0\leq r\leq R$ there is a constant, $C_3=C_3(n,\kappa,R)$ so
$$
K_{n,\kappa}(t,r)\leq C_3K_{n,0}(t,r).
$$
Moreover, because $|i(x)-i(x_0)|\leq \dist_g(x,x_0)$ and $\partial_r K_{n,0}\leq 0$, 
$$
K_{n,0}(t,\dist_g(x,x_0))\leq G(t,x).
$$
Hence, on $U$ we have
$$
|H(t,x)-G(t,x)|\leq H(t,x)+G(t, x)\leq (1+C_3) G(t,x),
$$
which implies there is a $C_4$ so
\begin{align*}
\int_{\Sigma \cap (\mathcal{B}_R^g(x_0) \setminus \mathcal{B}^g_{t^{1/3}}(x_0))} |H(t,x)-G(t,x)|&\leq C_4 t^{-\frac{n}{2}}e^{-\frac{1}{t^{1/3}}} Vol_g(\mathcal{B}_R^g(x_0)\cap \Sigma  )\\
&\leq C_4' t^{1/3} \Theta(t).
\end{align*}
Combining the estimates on $\mathcal{B}^g_{t^{1/3}}(x_0)$ and on $\mathcal{B}_R^g(x_0) \setminus \mathcal{B}^g_{t^{1/3}}(x_0)$ we obtain the first bound for an appropriate choice of $C$.

Finally, that $\Theta(t)$ is bounded independent of $t$ is a straightforward consequence of $i(\Sigma)$ being a properly embedded submanifold -- see \cite[Theorem 9.1]{WhiteBoundaryMCF}. 
\end{proof}

We will also need the fact that $K_{n,\kappa}$ can be used at small times to give an approximation of the identity.
\begin{prop}\label{ApproxIdentProp}
	For $k\geq 0$, let $(M,g)$ be an $(n+k)$-dimensional complete Riemannian manifold (not necessarily Cartan-Hadamard) and let $\Sigma\subset M$ be a smooth properly embedded $n$-dimensional submanifold.  If $f\in C^0_c(\Sigma)$ is a compactly supported function and $x_0\in \Sigma$, then
	$$
	\lim_{t\to 0^+} \int_{\Sigma} f(x) H(t,x) dVol_\Sigma(x)=f(x_0)
	$$
	where
	 $$
	H(t,x)=K_{n,\kappa}(t, \dist_g(x,x_0)).
	$$
	If, in addition, $\Sigma$ is complete, properly embedded, and has exponential volume growth, then  the same holds when $f\in C^0(\Sigma)\cap L^\infty(\Sigma)$. 
\end{prop}
\begin{proof}
	Pick $r>0$ sufficiently small so $U=\mathcal{B}^g_r(x_0)$ is geodesically convex and precompact.  Let $i: (U,g)\to (\Real^d, g_{\Real})$ be some choice of isometric embedding.  It follows from \eqref{KnDecayEstUB} that $H(t,x)\to 0$ uniformly as $t\to 0$ in $M\setminus \mathcal{B}_{r/2}^g(x_0)$. Let $\psi$ be a smooth bump function supported in $\mathcal{B}_{r/2}^g(x_0)$ with $\psi=1$ in a neighborhood of $x_0$. Since $f$ is compactly supported, one has
	$$
	 \lim_{t\to 0^+} \int_{\Sigma} H(t,x) f(x) dVol_\Sigma(x)=\lim_{t\to 0^+} \int_{\Sigma} H(t,x) \psi (x) f(x) dVol_\Sigma(x)
	 $$
	 provided both limits exist.  As such we may assume $f$ has compact support in $U$.
	 By Lemma \ref{ApproxEuclidHKLem} this reduces to a question on submanifolds of Euclidean space where the existence and value of the above limit can be readily established.
	 
	 It remains to treat the case where $f$ is bounded and $\Sigma$ has exponential volume growth.  In this case, the exponential volume growth and \eqref{KnDecayEstUB} imply
	 $$
	 \lim_{t\to 0^+}\int_{\Sigma}  H(t,x)(1-\psi(x)) f(x) dVol_\Sigma(x)=0
	 $$
	 and the result follows from the compactly supported case.
\end{proof}

We are now ready to prove Theorem \ref{ConstCurvCharThm}.
\begin{proof}[Proof of Theorem \ref{ConstCurvCharThm}]
  Set $P_{\kappa}(t,x,y)= K_{n, \kappa} (t,d_g(x,y))$.  Proposition \ref{ApproxIdentProp} implies that,  for any $\phi\in C^0_{c}(M)$, one has
  $$
  \phi(x)=\lim_{\tau\to 0^+} \int_{M} P_{\kappa}(\tau,x,y) \phi(y) dVol_{g}(y).
  $$
  Fix a domain $D\in \mathcal{D}(M)$ and let $ p_{D}(t,x,y) $ be the Dirichlet heat kernel extended by $0$ outside $D\times D$.  The definition of heat kernel likewise ensures that for $x\in D$ and $\psi\in C^0(M)\cap L^\infty(M)$, one has
  $$
  \psi(x) =\lim_{\tau\to 0^+} \int_{M} p_{D}(\tau,x,y) \psi(y) dVol_{g}(y).
  $$
  Hence, for any $\tau>0$ and $x,y\in D$, 
  $$
  p_{D}(\tau,x,y)-P_{\kappa}(\tau,x,y)=\lim_{\epsilon\to 0^+} \int_{M} \int_{\epsilon}^{\tau-\epsilon} \frac{d}{ds}\left (p_{D}(s,x,z) P_{\kappa}(\tau-s,z,y) ds \right)dVol_{g}(z).
  $$
  Hence, by Fubini's theorem and the fact that $p_D$ is a fundamental solution in $D$ one has,  for $x,y\in D$ and $\tau>0$, that 
  \begin{align*}
  &	p_{D}(\tau,x,y)-P_{\kappa}(\tau,x,y)\\
  &=\lim_{\epsilon\to 0^+} \int_{\epsilon}^{\tau-\epsilon} \int_{M} \bigg[ P_{\kappa}(\tau-s,z,y)\partial_t p_D (s,x,z) 
    - p_D(s,x,z) \partial_t P_{\kappa} (\tau-s,z,y) \bigg]dVol_{g}(z) ds\\
    & = \lim_{\epsilon\to 0^+} \int_{\epsilon}^{\tau-\epsilon}  \int_{\bar{D}}  \bigg[P_{\kappa}(\tau-s,z,y) 
    \Delta_g^z p_D (s,x,z) 
    - p_D(x,z,s) \partial_t P_{\kappa} (\tau-s,z,y)\bigg] dVol_{g}(z) ds.\\
      \end{align*}
      Integrating by parts and using that $p_D(\tau,x,y)\equiv 0$ for $y\in \partial D$ while $\partial_{\nu} p_D (\tau,x,y)\leq 0$ for $y \in \partial D$ and $\nu$ the outward normal to $D$ yields, for $\tau>0$,
      \begin{align*}
p_{D}(\tau,x,y)-P_{\kappa}(\tau,x,y)&\leq \lim_{\epsilon\to 0^+} \int_{\epsilon}^{\tau-\epsilon} \int_{M} -(\partial_t -\Delta_{g}^z) P_{\kappa}(\tau-s,z,y) dVol_{g}(z)\\     
	&\leq 0 
	\end{align*}
where the second inequality follows from Lemma \ref{SupersolutionLem}.  As $P_{\kappa}(\tau,x,y)>0$ for $\tau>0$, it follows that
$$
p_{D}(\tau,x,y)\leq P_{\kappa}(\tau,x,y)
$$
for all $x,y\in M$ and $\tau>0$.

As $D\in \mathcal{D}(M)$ was arbitrary, we conclude that,
$$
p_{M}(t,x,y)\leq P_{\kappa}(t,x,y)
$$
for all $t=\tau>0$ and $x,y\in M$. 

For fixed $y_0\in M$, set
$$
u(t,x)= P_{\kappa}(t,x,y_0)-p_{M}(t,x,y_0)\geq 0.
$$
By Lemma \ref{SupersolutionLem}, $u$ is a supersolution of the heat equation, i.e., 
$$
(\partial_t -\Delta_{g}) u \geq 0.
$$
Moreover,  the inequality is strict at some point unless $(M,g)$ has constant sectional curvature $-\kappa^2$.
In particular, as $M$ is connected, by the strong maximum principle, either $u(t, x)>0$ for all $x\in M$ and $t>0$ or $u(t, x)\equiv 0$.  In the latter case, $u=P_{\kappa}$ is a solution of the heat equation and so $(M,g)$ has constant sectional curvature $-\kappa^2$.  If this occurs, then the Ricci curvature is bounded from below and there is a unique fundamental solution of the heat equation.  Moreover,  $p_M$ and $P_{\kappa}$ are both fundamental solutions and so $p_{M}=P_{\kappa}$ is the heat kernel of $(M,g)$.
\end{proof}	
We record the following corollary which could also be proved by elementary means.
\begin{cor}\label{KnkappaIneqCor}
	 For $n\geq 2$, $ K_{n, \kappa_0}(t,r)<K_{n, \kappa}(t,r)$ for $0\leq \kappa<\kappa_0 $ and all $r\geq 0, t>0$.
\end{cor}
\begin{rem}
	It is immediate from the definition that $K_{1, \kappa}=K_{1, 0}$ for all $\kappa\geq 0$.
\end{rem}
\begin{proof}
Observe that the space form $(\mathbb{H}^n(\kappa_0), g_{\mathbb{H}(\kappa_0)})$ with constant curvature $-\kappa^2_0$ is a Cartan-Hadamard manifold with sectional curvature bounded above by $-{\kappa}^2$.  Moreover,  as $((\mathbb{H}^n(\kappa_0),g_{\mathbb{H}(\kappa_0)})$ is complete and has constant Ricci curvature, the space admits a unique heat kernel.  In particular, for all $t>0$ and $x,y\in M$
$$
p_{\mathbb{H}^n(\kappa_0)}(t,x,y)=K_{n, \kappa_0}(t,d_{\mathbb{H}(\kappa_0)}(x,y)).
$$
Hence,  Theorem \ref{ConstCurvCharThm} implies that,  for $t>0$ and all $x,y\in \mathbb{H}^n(\kappa_0)$,
$$
K_{n, {\kappa}_0}(t,d_{\mathbb{H}(\kappa_0)}(x,y))=p_{\mathbb{H}^n(\kappa_0)}(t,x,y)<K_{n, {\kappa}}(t,d_{\mathbb{H}(\kappa_0)}(x,y)).
$$
Clearly, $d_{\mathbb{H}(\kappa_0)}(x,y)$ achieves every value in $[0, \infty)$ and the result follows.
\end{proof}

\section{Huisken Monotonicity in Cartan-Hadamard manifolds}
Let $(M,g)$ be an $(n+k)$-dimensional Cartan-Hadamard manifold with $\mathrm{sec}_g\leq -\kappa^2 \leq 0$.  We establish a form of Huisken monotonicity for well-behaved mean curvature flows in $M$. To that end, set
$$
\Phi_{n, \kappa}^{t_0,x_0}(t,x)=K_{n,\kappa}(t_0- t,  \dist_{g}(x,x_0))=K_{n,\kappa}(t_0- t,  \rho(x)).
$$
When $k=0$, this is the backwards heat kernel on $(\mathbb{H}^n(\kappa), g_{\mathbb{H}(\kappa)})$, the space form of constant curvature $-\kappa^2$.
Observe that, for $x\neq x_0$, 
\begin{align*}
	\partial_\rho  \Phi_{n,\kappa}^{t_0, x_0}(t,x)= \partial_r K_{n, \kappa} (t_0-t,\rho(x))=\partial_\rho K_{n, \kappa}\leq 0,
\end{align*}
where the second equality is a slight abuse of notation.  Likewise,
\begin{align*}
	\partial_\rho^2 \Phi_{n,\kappa}^{t_0, x_0}(t,x)= \partial_r^2  K_{n, \kappa} (t_0-t,\rho(x))=\partial_\rho^2 K_{n, \kappa}.
\end{align*}
Computing as in Lemma \ref{SupersolutionLem} and using the above notion we obtain
\begin{align*}
\left(\frac{\partial}{\partial t} +\Delta_{g}\right) \Phi_{n,\kappa}^{t_0,x_0} &=-\partial_t K_{n,\kappa} +\partial_\rho^2 K_{n,\kappa}+ \partial_\rho K_{n,\kappa} \Delta_g \rho\\
 &= -(n-1) \ct_{\kappa}(\rho)\partial_\rho K_{n,\kappa}+ \partial_\rho K_{n,\kappa} \Delta_g \rho.
\end{align*}
We may rewrite the second equality as
\begin{align}\label{BackwardHeatPhink}
	\left(\frac{\partial}{\partial t} +\Delta_{g}\right) \Phi_{n,\kappa}^{t_0,x_0}=\left(\Delta_g \rho-(n-1) \ct_{\kappa} (\rho)\right) \partial_\rho \Phi_{n,\kappa}^{t_0, x_0}.
\end{align}
While we do not use it in what follows,  we remark that the estimates of Proposition \ref{HessCompProp} imply that 
\begin{align*}
	\left(\frac{\partial}{\partial t} +\Delta_{g}\right)\Phi_{n,\kappa}^{t_0,x_0}\leq  k ct_{\kappa}(\rho) \partial_\rho \Phi_{n,\kappa}^{t_0,x_0}\leq 0.
\end{align*}
Moreover, the first inequality is strict at some point unless $(M,g)$ is $(\mathbb{H}^{n+k}(\kappa),  g_{\mathbb{H}(\kappa)})$ and the 
 second inequality is strict  unless $k=0$.

We compute the change of the integral of the kernel $\Phi_{n,\kappa}^{t_0,x_0}$ along a well-behaved mean curvature flow when paired with a compactly supported test function. 
\begin{prop}\label{MonotonicityFormulaProp}
	Suppose $t\in [T_1,T_2)\mapsto \Sigma_t$ is a locally smooth mean curvature flow in $M$ of smooth properly embedded $n$-dimensional submanifolds. For any $t_0\in (T_1,T_2]$, $x_0\in M$, $t\in [T_1,t_0)$ and time varying test function $\psi\in C^{\infty}_c([T_1,T_2]\times M)$ one has
	\begin{align*}
	\frac{d}{dt} \int_{\Sigma_t} & \psi \Phi_{n,\kappa}^{t_0,x_0} dVol_{\Sigma_t}=\int_{\Sigma_t} \left( \frac{d}{dt}\psi -\Delta_{\Sigma_t}\psi \right) \Phi_{n,\kappa}^{t_0,x_0} dVol_{\Sigma_t}\nonumber\\
	 &- \int_{\Sigma_t} \psi \left( \left| \frac{\nabla^\perp_{\Sigma_t} \Phi_{n,\kappa}^{t_0,x_0}}{\Phi_{n,\kappa}^{t_0,x_0}} -\mathbf{H}_{\Sigma_t}\right|^2 + Q^{t_0,x_0}_{n,\kappa}( t, x,N_x \Sigma_t) \right) \Phi_{n,\kappa}^{t_0,x_0} dVol_{\Sigma_t}.\label{prop5.1}
	\end{align*}
	Here,
	$$
	Q^{t_0,x_0}_{n,\kappa}( t,x, E) =\sum_{i=1}^{k}\nabla^2_{g}\log \Phi_{n,\kappa}^{t_0, x_0} (E_i, E_i) + \left((n-1) \ct_{\kappa}(\rho) -\Delta_g \rho\right)\partial_\rho \log \Phi_{n,\kappa}^{t_0, x_0} 
	$$
  where $\rho(x)=\dist_g(x,x_0)$, $E$ is a $k$-dimensional subspace of $T_xM$ and $\set{E_1, \ldots, E_{k}}$ is an orthonormal basis of $E$.
\end{prop}
\begin{proof}
	Recall, that for any $f\in C^1([T_1,T_2)\times M)$ and $x\in \Sigma_t$ we define
	$$
	\frac{d}{d t} f(t,x)=\frac{\partial}{\partial t} f (t,x)+\mathbf{H}_{\Sigma_t} \cdot \nabla_{g} f(t,x).
	$$
As in \cite[Chapter 3]{Ecker2001},  first variation formula yields, 
	\begin{align*}
		\frac{d}{dt} \int_{\Sigma_t} &\psi  \Phi_{n,\kappa}^{t_0, x_0} dVol_{\Sigma_t}= \int_{\Sigma_t} \bigg[\frac{d}{d t} \bigg( \psi \Phi_{n,\kappa}^{t_0,x_0}\bigg) - \psi  \Phi_{n,\kappa}^{t_0, x_0}|\mathbf{H}_{\Sigma_t}|^2\bigg] dVol_{\Sigma_t}\\		
		&=\int_{\Sigma_t} \bigg[\Phi_{n,\kappa}^{t_0, x_0}\frac{d}{d t}\psi+ \psi\frac{d}{d t} \Phi_{n,\kappa}^{t_0, x_0} -\psi|\mathbf{H}_{\Sigma_t}|^2 \Phi_{n,\kappa}^{t_0, x_0}\bigg]  dVol_{\Sigma_t}\\
	&	=\int_{\Sigma_t}\bigg[
		\Phi_{n,\kappa}^{t_0, x_0}\bigg(\frac{d}{d t}-\Delta_{{\Sigma_t}}\bigg)\psi+\psi\bigg(\frac{d}{d t}+\Delta_{{\Sigma_t}}\bigg)\Phi_{n,\kappa}^{t_0, x_0}
	 -\psi|\mathbf{H}_{\Sigma_t}|^2 \Phi_{n,\kappa}^{t_0, x_0}\bigg]  dVol_{\Sigma_t}
	\end{align*}
	where the last equality follows from integrating by parts and we used throughout that $\Sigma_t$ is properly embedded and $\psi(t,\cdot)$ has compact support. 
	
	Standard geometric computations imply that on $\Sigma_t$,
	$$
	\Delta_{g} \Phi_{n,\kappa}^{t_0, x_0} =  \Delta_{{\Sigma_t}}\Phi_{n,\kappa}^{t_0, x_0} -\mathbf{H}_{\Sigma_t}\cdot \nabla_{g}  \Phi_{n,\kappa}^{t_0, x_0}+\sum_{i=1}^{k}\nabla^2_{g }\Phi_{n,\kappa}^{t_0, x_0} (E_i, E_i)
	$$
	where $\set{E_1, \ldots, E_{k}}$ is an orthonormal basis of $N_x\Sigma_t\subset T_{x}M$.  
	Hence,
	\begin{align*}
		\left(\frac{\partial}{\partial t} +\Delta_{\Sigma_t}\right)\Phi_{n,\kappa}^{t_0, x_0} & =	\left(\frac{\partial}{\partial t} +\Delta_{g}\right)\Phi_{n,\kappa}^{t_0, x_0}+\mathbf{H}_{\Sigma_t}\cdot \nabla_{g}  \Phi_{n,\kappa}^{t_0, x_0}
		-\sum_{i=1}^{k}\nabla^2_{g }\Phi_{n,\kappa}^{t_0, x_0} (E_i, E_i).
	\end{align*}
	It follows from \eqref{BackwardHeatPhink} that
	\begin{equation}\label{BackwardSubmanifoldEq} \begin{aligned}
		\left(\frac{\partial}{\partial t} +\Delta_{\Sigma_t} \right)&\Phi_{n,\kappa}^{t_0, x_0} 	= \left(\Delta_g \rho-(n-1) \ct_{\kappa} (\rho)\right) \partial_\rho \Phi_{n,\kappa}^{t_0, x_0} \\
		&+\mathbf{H}_{\Sigma_t}\cdot \nabla_{{g}}  \Phi_{n,\kappa}^{t_0, x_0} -\sum_{i=1}^{k}\nabla^2_{g }\Phi_{n, \kappa}^{t_0, x_0} (E_i, E_i)\\
		&=\mathbf{H}_{\Sigma_t}\cdot \nabla_{g}  \Phi_{n,\kappa}^{t_0, x_0} -Q_{n,\kappa}^{t_0,x_0}(t,x,N_x\Sigma_t) \Phi_{n,\kappa}^{t_0, x_0}	-\frac{|\nabla_{\Sigma_t}^\perp \Phi_{n,\kappa}^{t_0,x_0}|^2}{\Phi_{n,\kappa}^{t_0, x_0}},
	\end{aligned}
\end{equation}
	where the second equality uses
	$$
	\nabla_{\Sigma_t}^2 \log \Phi_{n,\kappa}^{t_0,x_0}=\frac{\nabla_{\Sigma_t}^2\Phi_{n,\kappa}^{t_0,x_0}}{\Phi_{n,\kappa}^{t_0,x_0}} -\frac{d\Phi_{n,\kappa}^{t_0,x_0} \otimes d\Phi_{n,\kappa}^{t_0,x_0}}{(\Phi_{n,\kappa}^{t_0,x_0})^2}.
	$$
	Hence, by putting everything together, one has
	\begin{align*}
		\frac{d}{dt} \int_{\Sigma_t}& \psi  \Phi_{n,\kappa}^{t_0,x_0} dVol_{\Sigma_t}
		=\int_{\Sigma_t}\bigg[ \frac{d}{dt} \psi-\Delta_{\Sigma_t}\psi \bigg] \Phi_{n,\kappa}^{t_0,x_0} dVol_{\Sigma_t}
		\\
		&+\int_{\Sigma_t}\psi  \bigg[ \left(\frac{\partial}{\partial t} +\Delta_{\Sigma_t} \right)\Phi_{n,\kappa}^{t_0, x_0}+\nabla_{g} \Phi_{n,\kappa}^{t_0, x_0} \cdot \mathbf{H}_{\Sigma_t}-|\mathbf{H}_{\Sigma_t}|^2\Phi_{n,\kappa}^{t_0, x_0} \bigg] dVol_{\Sigma_t}\\
		&=\int_{\Sigma_t}\bigg[ \frac{d}{dt} \psi-\Delta_{\Sigma_t}\psi \bigg] \Phi_{n,\kappa}^{t_0,x_0} dVol_{\Sigma_t}\\
		&- \int_{\Sigma_t}\psi \left( \left|\frac{\nabla_{\Sigma_t}^\perp \Phi_{n,\kappa}^{t_0,x_0}}{\Phi_{n,\kappa}^{t_0,x_0}} -\mathbf{H}_{\Sigma_t} \right|^2 +Q_{n,\kappa}^{t_0,x_0}(t,x,N_x\Sigma_t) \right) \Phi_{n,\kappa}^{t_0,x_0}dVol_{\Sigma_t}.
	\end{align*}
	This completes the proof.
\end{proof}

We use the super-convexity estimate of Proposition \ref{KnProp} and the Hessian comparison of Proposition \ref{HessCompProp} to see that $Q_{n,\kappa}^{t_0,x_0}$ has a sign.
\begin{lem}\label{QPosLem}
Fix $x_0\in M$ and $t_0\in \Real$ and choose $x\in M\setminus \set{x_0}$, $t<t_0$ and a $k$-dimensional subspace $E\subset T_x M$.  The inequality
$$Q_{n,\kappa}^{t_0,x_0}(t,x,E)   \geq 0
$$
holds, where
$$
Q_{n,\kappa}^{t_0,x_0}(t,x,E) =\sum_{i=1}^{k}\nabla^2_{g}\log \Phi_{n,\kappa}^{t_0, x_0} (E_i, E_i) +\left((n-1) \ct_\kappa( \rho) -\Delta_g \rho\right)\partial_\rho \log \Phi_{n,\kappa}^{t_0, x_0},$$
 $\rho=\dist_g(\cdot,x_0)$ and $E_1, \ldots, E_{n+k}$ is an orthonormal basis of $T_x M$ with $E=\mathrm{span}\set{E_1 \ldots, E_k}$.  Moreover, the inequality is strict unless both:
\begin{enumerate}
	\item 
	\label{StrictCase1} $\sec_{g}(E_i(t)\wedge \gamma'(t))=-\kappa^2$, where $\gamma:[0, \rho(x)]\to M$ is the unique minimizing geodesic connecting $x_0$ and $x$,  $E_i(t)$ is the parallel transport of $E_i$ along $\gamma$ and $k+1\leq i \leq n+k$ is such that $\gamma'(\rho(x))=\nabla_g \rho(x)$ is not parallel to $E_i$;
	\item \label{StrictCase2} Either $\kappa=0$ or $g(E_i, \nabla_g \rho(x))=0$, for all $1\leq i \leq k$.
\end{enumerate}	
\end{lem}
\begin{proof}
	We first compute that, for $\kappa>0$, 
	\begin{align*}
		\partial_r^2 \log K_{n,\kappa} -\ct_{\kappa}(r) \partial_r \log K_{n, \kappa} &=
		\kappa^{2+n} \left(  \partial^2_{r} \log K_n (t, \kappa r )-\coth(\kappa r) \partial_r \log K_n(t,\kappa r)\right).
	\end{align*}
	Hence, by Proposition \ref{KnProp}, one has, for $\kappa>0$,
	$$
	\partial_r^2 \log K_{n,\kappa}(t,r) -\ct_{\kappa}(r) \partial_r \log K_{n, \kappa}(t,r) \geq 0
	$$
	with strict inequality when $r>0$.  When $\kappa=0$ and $r>0$ one directly computes,
	$$
	\partial_r^2 \log K_{n,0} -\ct_{0} (r) \partial_r \log K_{n,0}= \partial_r^2(-\log r) -\frac{1}{r} \partial_r( -\log r)=0.
	$$
	As such, for all $\kappa\geq 0$, 
	$$
	\partial_r^2 \log K_{n,\kappa} -\ct_{\kappa}(r) \partial_r \log K_{n, \kappa} \geq 0
	$$
	with equality only when $r=0$ or $\kappa=0$.

Abusing notation as above, we readily compute,
\begin{align*}
	\nabla^2_{g}\log \Phi_{n,\kappa}^{t_0, x_0}&=(\partial^2_\rho \log K_{n,\kappa}) d\rho^2 +(\partial_{\rho} \log K_{n, \kappa})\nabla_{g}^2 \rho.
	\end{align*}
	Hence, continuing to abuse notation,
	\begin{align*}
		 Q_{n,\kappa}^{t_0,x_0}(t,x,E) &= \sum_{i=1}^k (\partial^2_\rho \log K_{n, \kappa})  (d\rho(E_i))^2 + (n-1)( \partial_\rho\log  K_{n,\kappa}) \ct_{\kappa}(\rho)\\ 
		 &-( \partial_\rho\log  K_{n,\kappa}) \sum_{i=k+1}^{n+k} (\nabla^2_{g} \rho) (E_i,E_i) \\
		 &\geq  \sum_{i=1}^k (\partial^2_\rho \log K_{n, \kappa})  (d\rho(E_i))^2 + (n-1)( \partial_\rho\log  K_{n,\kappa}) \ct_{\kappa}(\rho)\\
		 &-( \partial_\rho\log  K_{n,\kappa})\sum_{i=k+1}^{n+k} \ct_{\kappa}(\rho)(g(E_i,E_i)-(d\rho(E_i))^2)
	\end{align*}
	where the inequality used the fact that $-\partial_\rho\log  K_{n,\kappa}\geq 0$ and Proposition \ref{HessCompProp}.  Moreover, as $-\partial_\rho\log  K_{n,\kappa}>0$ for $x\neq x_0$, Proposition \ref{HessCompProp} further implies that  this inequality is strict unless case \eqref{StrictCase1} holds.
	As
	$$
	1=|\nabla_g \rho|^2=\sum_{i=1}^{n+k} (d\rho(E_i))^2,$$
	we obtain
\begin{align*}
	Q_{n,\kappa}^{t_0,x_0}(t,x,E) &\geq (\partial^2_\rho \log K_{n, \kappa})  \sum_{i=1}^k (d\rho(E_i))^2 -\ct_{\kappa}(\rho)\partial_\rho\log  K_{n,\kappa} \\
	&+\ct_{\kappa}(\rho) \partial_\rho\log  K_{n,\kappa}\sum_{i=k+1}^{n+k}(d\rho(E_i))^2\\
	&=(\partial^2_\rho \log K_{n, \kappa}- \ct_{\kappa}(\rho) \partial_\rho\log  K_{n,\kappa}) \sum_{i=1}^k (d\rho(E_i))^2\geq 0.
\end{align*}
Here the final inequality follows from the convexity property of $K_{n,\kappa}$.  One has equality in the final inequality only when $\kappa=0$ or when $\kappa>0$ and
$$
\sum_{i=1}^k (d\rho(E_i))^2=0.
$$
That is, only when $\kappa=0$ or when $g(E_i, \nabla_g \rho)=0$ for all $1\leq i \leq k$.
\end{proof}
We will need the following characterization of geodesic cones in Cartan-Hadamard manifolds.
\begin{lem}\label{ConeLem}
	Let $\Sigma \subset M$ be a complete and proper $n$-dimensional submanifold and  $x_0\in M$.  If $\rho=\dist_{g}(\cdot,x_0)$ and $|\nabla_\Sigma^\perp \rho|=0$ on $ \Sigma\setminus \set{x_0}$, then $\Sigma$ is a geodesic cone over $x_0$.
\end{lem}
\begin{proof}
	First observe that, $\Sigma$ is connected and $x_0\in \Sigma$.  Indeed, if $\Sigma$ was not connected or $x_0\not\in \Sigma$, then there would be a component $\Sigma'$ that does not contain $x_0$.  As $\Sigma$ is proper, this implies the existence of a point $x'\in \Sigma'$ at which $\rho$ achieves its (non-zero) minimum.  As this is a regular point of $\rho$, one must have $\nabla_{\Sigma}\rho=0$ at this point and so $|\nabla^\perp_\Sigma \rho|(x')=1$, a contradiction.

Let $f_\kappa(r)$ be given by \eqref{fkEqn}.  We have
$$
\nabla^2_g f_{\kappa}(\rho)=f'_{\kappa}(\rho) \nabla^2_g \rho +(1+\kappa^2 f_{\kappa}(\rho)) d\rho^2. 
$$
Using $f_{\kappa}'\geq 0$ and the lower bound on $\nabla_g^2\rho$ given by Proposition \ref{HessCompProp} yields
$$
\nabla^2_g f_{\kappa}(\rho)\geq (1+\kappa^2 f_{\kappa}(\rho)) g.
$$
Hence, the hypothesis that $\nabla_\Sigma^\perp \rho=0$ ensures $\nabla_\Sigma^\perp f_{\kappa} (\rho)=0$ and
	$$
	\nabla_{\Sigma}^2 f_{\kappa}(\rho)\geq (1+\kappa^2 f_{\kappa}(\rho)) g_{\Sigma}.
	$$
	
	As $\Sigma$ is connected, complete and contains $x_0$, for each $x\in \Sigma\setminus \set{x_0}$, there is a $\gamma:[0,L]\to \Sigma$ which is a minimizing geodesic in $\Sigma$ parameterized by arclength that connects $x_0$ to $x$. That is, $\gamma$ satisfies $\gamma(0)=x_0$ and $\gamma'(s)$ is of unit length. 
	One computes, 
	$$
	\frac{d^2}{ds^2} f_{\kappa}(\rho(\gamma(s)))=\nabla^2_{\Sigma} f_{\kappa}(\rho) (\gamma'(s), \gamma'(s))\geq 1+\kappa^2 f_{\kappa}(\rho(\gamma(s))).
	$$
	Set
	$$
	g(s)=f_{\kappa}(\rho(\gamma(s)))
	$$
	so we have $g(0)=f_{\kappa}(0)=0$ and $g'(0)=f_{\kappa}'(0)=0$ and 
	$$
	g''(s)\geq 1+\kappa^2 g(s).
	$$
As $f_{\kappa}''=1+\kappa^2 f_{\kappa}$, standard ODE comparison implies that
	$$
	f_{\kappa}(\rho(\gamma(s)))=g(s)\geq f_{\kappa}(s).
	$$
	As $f_{\kappa}$ is strictly monotone on the positive reals it follows that, $\rho(\gamma(s))\geq s$.  This means 
\begin{equation}\label{EqualGeo}
	\dist_g(x,x_0)=\rho(x)=\rho(\gamma(L))\geq L= \dist_{\Sigma}(x,x_0).
\end{equation}
	This is only possible if $\gamma$ is also a minimizing geodesic in $(M,g)$ and we have equality throughout \eqref{EqualGeo}.  In particular, $\gamma$ is the unique unit speed geodesic connecting $x_0$ to $x$ and lies entirely within $\Sigma$.   As $x$ is arbitrary, $\Sigma$ must be a cone over $x_0$.
\end{proof}

\begin{thm}\label{HuiskenMonThm}
	Suppose $t\in [T_1,T_2)\mapsto \Sigma_t$ is a locally smooth mean curvature flow in $M$ of properly embedded $n$-dimensional submanifolds with exponential volume growth.   For any $t_0\in (T_1,T_2]$, $x_0\in M$, and $T_1\leq t_1\leq t_2<t_0$ if  $Ric_g \geq -\Lambda^2 g$ on $\Sigma_t$, for all $t\in [t_1,t_2]$ and a uniform $\Lambda\geq 0$, then
	\begin{align*}
\int_{\Sigma_{t_2}} \Phi_{n,\kappa}^{t_0,x_0} (t, x)dVol_{\Sigma_{t_2}} (x)  &\leq \int_{\Sigma_{t_1}} \Phi_{n,\kappa}^{t_0,x_0} (t, x)dVol_{\Sigma_{t_1}} (x).
	\end{align*}
	When $\kappa>0$, the inequality is strict unless $\Sigma_{t}$ is a geodesic cone over $x_0$ that is a totally geodesic submanifold isometric to $\mathbb{H}^{n}({\kappa})$.
\end{thm}

\begin{rem}\label{RicCondRem}
	If all $\Sigma_t$ remain in a fixed compact subset of $M$, then there automatically exists such a constant $\Lambda$.  For instance, this is automatic if one has a classical mean curvature flow of a closed hypersurface on a compact time interval.
	
	In general, without a lower bound on the Ricci curvature, the result does not hold.  For instance, consider the metric in polar coordinates on $M=\Real^2$
	$$
	g=dr^2 +r^2 e^{\frac{1}{2}r^4}d\theta^2.
	$$
Computing the Gauss curvature one has
	$$
	K(r,\theta)= - (5r^2+r^6)\leq 0.
	$$
	Clearly, all geodesics rays from $0$ are of infinite length and so the metric is complete by the Hopf-Rinow theorem.  Hence,
$(\Real^2,g)$ is a Cartan-Hadamard manifold. 
	Moreover,  the curves $\sigma_R=\partial \mathcal{B}_R^g(0)$ for $R>0$ are of constant geodesic curvature (with respect to the outward normal) and have curvature
	$$
	H_{\sigma_R}=\frac{1}{R}+R^3.
	$$
	 It follows that if, 
	 $$
	 t\mapsto R(t)=\sqrt{\tan(-2t)}
	 $$
	 for $t\in \left( -\frac{\pi}{4},0\right) $, then
	 $$
	 t\mapsto  \partial \mathcal{B}_{R(t)}^g(0)
	 $$
	 is a classical mean curvature flow. 
	  However, as $\lim_{t\to -\frac{\pi}{4}^+} R(t)=\infty$, we can create a locally smooth mean curvature flow that is not a classical flow by extending the time of the flow to $(-\infty, 0)$ by setting
	 $$
	 t\mapsto \Sigma_t =\left\{\begin{array}{cc} \emptyset & t\leq  -\frac{\pi}{4} \\ \partial \mathcal{B}_{R(t)}^g(0)& t\in\left( -\frac{\pi}{4},0\right). \end{array} \right.
	  $$
	  This is an example of a flow ``blossoming from infinity" mentioned in \cite{Ilmanen1993} and is also one for which the monotonicity formula  does not hold across $t=-\frac{\pi}{4}$.
\end{rem}
\begin{proof}
	To verify the inequality we use Proposition \ref{MonotonicityFormulaProp} and Lemma \ref{QPosLem} together with appropriately chosen test functions $\psi_R$.
First pick a smooth cutoff $\chi:(-\infty, \infty)\to [0,1]$ 
 so $\chi (s)=1$ for $s\leq 1$, $\chi(s)=0$ for $s\geq 2$, $\chi'\leq 0$, and
 $$
  |\chi'|+|\chi''|\leq 100.
 $$
 For $R\geq 1$, let  $\chi_R(s)=\chi(R^{-1}s)$. It follows that $\chi_R$ is supported in $(-\infty,2R]$, $\chi_R'\leq 0$, $\chi_R'$ and $\chi_R''$ are supported in $[R,2R]$, and 
 $$
 R|\chi_R'|+R^2 |\chi''_R|\leq 100\chi_{2R}.
 $$
 Fix a point $x_0\in M$ and, for $R\geq 1$, let
  $$\psi_R(x)=\chi_R(\rho(x))=\chi_R( \dist_g(x,x_0)).$$
  Observe that this is compactly supported in $\bar{\mathcal{B}}_{2R}^g(x_0)$.
  We compute that
		$$
	\Delta_{\Sigma_t} \psi_R =  \Delta_{g}\psi_R +\mathbf{H}_{\Sigma_t}\cdot \nabla_{g}  \psi_R-\sum_{i=1}^{k}\nabla^2_{g }\psi_R (E_i, E_i),
	$$ 
		where $\set{E_1, \ldots, E_{k}}$ is an orthonormal basis of $N_x\Sigma_t\subset T_{x}M$. 
Hence, on $\Sigma_t$
$$
\left(\frac{d}{dt}-\Delta_{\Sigma_t} \right) \psi_R		=-\Delta_{g}\psi_R +\sum_{i=1}^{k}\nabla^2_{g }\psi_R (E_i, E_i).
$$
One computes that
$$
\nabla_g^2 \psi_R= \chi''_R d\rho \otimes d\rho + \chi'_R \nabla^2_g \rho.
$$
By the Laplace comparison theorem (e.g.,  \cite{Schoen1994}), the lower bound on Ricci curvature, the Hessian comparison theorem, the fact that $g$ is Cartan-Hadamard, and $\chi'\leq 0$, we conclude that on $\Sigma_t$, for $R\geq 1$,
\begin{align*}
\left(\frac{d}{dt}-\Delta_{\Sigma_t} \right) \psi_R& \leq - \chi''_R|\nabla_{\Sigma_t} \rho|^2-(n-1)  \chi'_R \ct_{\Lambda}(\rho) -\chi'_R\ct_0(\rho) (k-|\nabla^\perp_{\Sigma_t} \rho|^2)\\
 &\leq 100\left( R^{-2} +(n-1) R^{-1} \Lambda +k R^{-2}\right) \chi_{2R}\\
 &\leq CR^{-1}\psi_{2R}
\end{align*}
where the second inequality used that $\chi'_R$ and $\chi''_R$ are supported on $[R,2R]$ and $C=(n, k,\Lambda)$. 

%

For each $R\geq 1$, Proposition \ref{MonotonicityFormulaProp} implies
		\begin{align*}
		\frac{d}{dt} \int_{\Sigma_t} & \psi_R \Phi_{n,\kappa}^{t_0,x_0} dVol_{\Sigma_t}=\int_{\Sigma_t} \left(\left( \frac{d}{dt} -\Delta_{\Sigma_t} \right)\psi_R \right)\Phi_{n,\kappa}^{t_0,x_0} dVol_{\Sigma_t}\nonumber\\
		&- \int_{\Sigma_t} \psi_R \left( \left| \frac{\nabla^\perp_{\Sigma_t} \Phi_{n,\kappa}^{t_0,x_0}}{\Phi_{n,\kappa}^{t_0,x_0}} -\mathbf{H}_{\Sigma_t}\right|^2 + Q^{t_0,x_0}_{n,\kappa} \right) \Phi_{n,\kappa}^{t_0,x_0} dVol_{\Sigma_t}.\label{prop5.1}
	\end{align*}
	Integrating and using the above estimate with the fact that $\psi_R$ is non-negative yields
	\begin{equation*}
	\begin{aligned}
		 \int_{\Sigma_{t_2}}  &\psi_R \Phi_{n,\kappa}^{t_0,x_0} dVol_{\Sigma_{t_2}}+\int_{t_1}^{t_2} \int_{\Sigma_t} \psi_R \left( \left| \frac{\nabla^\perp_{\Sigma_t} \Phi_{n,\kappa}^{t_0,x_0}}{\Phi_{n,\kappa}^{t_0,x_0}} -\mathbf{H}_{\Sigma_t}\right|^2 + Q^{t_0,x_0}_{n,\kappa} \right) \Phi_{n,\kappa}^{t_0,x_0} dVol_{\Sigma_t}dt \\	 
&\leq 	\int_{\Sigma_{t_1}} \psi_R  \Phi_{n,\kappa}^{t_0,x_0} dVol_{\Sigma_{t_1}}	 + C R^{-1}\int_{t_1}^{t_2} \int_{\Sigma_t} \psi_{2R}  \Phi_{n,\kappa}^{t_0,x_0} dVol_{\Sigma_t} dt.
	\end{aligned}
\end{equation*}
Applying the monotone convergence theorem and using the exponential volume growth of $\Sigma_t$, along with the fact that $\psi_R$ is bounded, for any $t\in [t_1,t_2]$ one has 
	$$
	\lim_{R\to \infty} \int_{\Sigma_t} \psi_{R}  \Phi_{n,\kappa}^{t_0,x_0} dVol_{\Sigma_t} =\lim_{R\to \infty}\int_{\Sigma_t} \psi_{2R}  \Phi_{n,\kappa}^{t_0,x_0} dVol_{\Sigma_t}=\int_{\Sigma_t}  \Phi_{n,\kappa}^{t_0,x_0} dVol_{\Sigma_t}<\infty.
	$$
	Hence, letting $R\to \infty$ and appealing to Lemma \ref{QPosLem} and the monotone convergence theorem, one concludes
		\begin{equation}\label{WeightedIneqEqn}
		\begin{aligned}
		0	&\leq \int_{t_1}^{t_2} \int_{\Sigma_t}\left( \left| \frac{\nabla^\perp_{\Sigma_t} \Phi_{n,\kappa}^{t_0,x_0}}{\Phi_{n,\kappa}^{t_0,x_0}} -\mathbf{H}_{\Sigma_t}\right|^2 + Q^{t_0,x_0}_{n,\kappa} \right) \Phi_{n,\kappa}^{t_0,x_0} dVol_{\Sigma_t}dt \\	 
	&\leq 	\int_{\Sigma_{t_1}}  \Phi_{n,\kappa}^{t_0,x_0} dVol_{\Sigma_{t_1}}- \int_{\Sigma_{t_2}}  \Phi_{n,\kappa}^{t_0,x_0} dVol_{\Sigma_{t_2}},
		\end{aligned}
\end{equation}
which immediately proves the main inequality.

\smallskip
	  It remains only to characterize the case of equality. 
	To that end, first observe that $Q_{n,\kappa}^{t_0,x_0}(t,x,N_{x}\Sigma_t)$ is continuous in $x\neq x_0$ and so if $Q_{n,\kappa}^{t_0,x_0}(t',x',N_{x'}\Sigma_t)>0$ at a point $x'\in \Sigma_{t'}\setminus \set{x_0}$ for $t'\in (t_1,t_2)$, then
	it is positive in a neighborhood of $x'$ in $\Sigma_t$.  In fact, it is positive in a spacetime neighborhood of $(t',x')$ and so
    \begin{align*}
0<\int_{t_1}^{t_2} \int_{\Sigma_t}\left( \left| \frac{\nabla^\perp_{\Sigma_t} \Phi_{n,\kappa}^{t_0,x_0}}{\Phi_{n,\kappa}^{t_0,x_0}} -\mathbf{H}_{\Sigma_t}\right|^2 + Q^{t_0,x_0}_{n,\kappa} \right) \Phi_{n,\kappa}^{t_0,x_0} dVol_{\Sigma_t}dt.
    \end{align*}
	Thus, equality can occur only when $Q_{n,\kappa}^{t_0,x_0}(t,x,N_{x}\Sigma_t)=0$ for all $x\in \Sigma_t\setminus \set{x_0}$ and $t\in (t_1,t_2)$.  As $\kappa>0$, Lemma \ref{QPosLem} implies that at all such $x$,
	$$
     g(\nabla_g \rho, E_i)=0
     $$
     where $E_i\in N_x\Sigma_t$.  That is,
     $$
     |\nabla_g^\perp \rho|=0
     $$
     on $\Sigma_t\setminus\set{x_0}$. By Lemma \ref{ConeLem},  $\Sigma_t$ is a cone over $x_0$ that is totally geodesic.   As $\Sigma_t$ is assumed to be smooth at $x_0$ it is actually a totally geodesic submanifold. Hence, $\Sigma_t$ is also minimal and so the flow is static. For notational simplicity let us set $\Sigma'=\Sigma_t$ and observe we can think of $\Sigma'_s=\Sigma$, $s\in \Real$ as a static solution to mean curvature flow which agrees with $\Sigma_t$ for $t\in (t_1,t_2)$. The minimality of $\Sigma'$ along with the fact that $\nabla_g^\perp \rho=0$ on $\Sigma'$ and the computation \eqref{BackwardSubmanifoldEq} imply that
     $$
     \left(\frac{d}{dt}+\Delta_{\Sigma'}\right) \Phi_{n, \kappa}^{t_0,x_0}=0.
	$$
	This means $H_{n, \kappa}(t,x)=\Phi_{n,\kappa}^{t_0,x_0} (-t, x) $ solves the linear heat equation exactly on $\Sigma'$. Again using the fact that $\Sigma'$ is totally geodesic the Gauss equations imply that
	$$
	\sec_{\Sigma'}(P)=\sec_g(P)\leq -\kappa^2
	$$
	for any two-plane $P\subset T_x \Sigma'\subset T_x M$.
	Hence, by Lemma \ref{SupersolutionLem}, as $H_{n, \kappa}(t,x)$ is an exact solution to the heat equation on $\Sigma'$ it must be the case that $\Sigma'$ has constant curvature $-\kappa^2$.  That is, $\Sigma_t=\Sigma'$ is a totally geodesic smooth cone that is isometric to $\mathbb{H}^n(\kappa)$, the space-form of constant curvature $-\kappa^2$.  
\end{proof}

As a consequence of this monotonicity, we may use the kernels $\Phi_{n, \kappa}^{t_0,x_0}$ to define Gaussian density for mean curvature flows in Cartan-Hadamard manifolds as well as conclude the existence of tangent flows modeled on (singular) self-shrinkers in Euclidean space.  Before stating the next corollary, for the convenience of the readers, we first recall the definition of Gaussian density given by White \cite{WhiteBoundaryMCF}.  For $n\geq 1$, let $G,\hat{G}: (-\infty,0)\times \Real^d \to \Real$ be defined by
\begin{align*}
	G(t,x)=K_{n,0}(-t,|{x}|)=\frac{1}{(-4\pi t)^{\frac{m}{2}}}e^{-\frac{|x|^2}{-4t}} \text{ and }
	\hat{G}(t,x)=\chi(|x|)G(t,x),
\end{align*}
where $\chi$ is smooth and compactly supported in $[0,1)$ with $\chi=1$ on $[0,\frac{1}{2}]$ and $\chi'\leq 0$. 
\begin{definition}\cite[Definition 6.3]{WhiteBoundaryMCF}\label{density}
Given a locally smooth mean curvature flow of properly embedded $n$-dimensional submanifolds 
$$\mathcal{M}=t\in [T_1,T_2)\to \Sigma_t\subset M$$
and a point $x_0\in M$ and a time $t_0\in (T_1,T_2]$, the \emph{Gaussian density} of $\mathcal{M}$ at $(t_0, x_0)$ is defined to be
$$
\Theta(\mathcal{M}; (t_0,x_0))=\lim_{t \to t_0^-}\int_{i(\Sigma_t)} \hat{G}(t,x) dVol_{i(\Sigma_t)}
$$
where $i:M\to \Real^{d}$ is an isometric embedding of $M$.  This value is well defined and independent of the choice of $\chi$ and $i$.
\end{definition}
\begin{cor}
\label{DensityLem}
Suppose $\mathcal{M}=t\in [T_1,T_2)\mapsto \Sigma_t$ is a locally smooth mean curvature flow in $M$ of $n$-dimensional submanifolds with exponential volume growth and $Ric_g \geq -\Lambda^2 g$ on $\Sigma_t$, for $t\in [t_1,t_2]$ and a uniform $\Lambda\geq 0$. For any $x_0\in M$ and $t_0\in (T_1, T_2]$, there is a well defined and finite limit
$$
\Theta_{\kappa}(t_0, x_0)=\lim_{t\to t_0^-} \int_{\Sigma_t} \Phi_{n,\kappa}^{t_0,x_0} (t, x)dVol_{\Sigma_t} (x)\leq \int_{\Sigma_{T_1}} \Phi_{n,\kappa}^{t_0,x_0}(t,x) dVol_{\Sigma_{T_1}}(x).
$$
Moreover, this quantity agrees with the Gaussian density of the flow at $(t_0,x_0)$.  As a consequence,  for any sequence of times $t_i\to t_0^-$, up to passing to a subsequence, there is a self-shrinking integer rectifiable $n$-varifold $V$ in $\Real^{n+k}$ so 
$$
\Theta_{\kappa}(t_0,x_0)=\lambda[V]=\int \Phi_{n, 0}^{0,0}(-1,x) d\mu_V(x)=\Theta(\mathcal{M}, (t_0,x_0))
$$
and $\Sigma_{t_i}$ converges to $V$ in an appropriate sense, i.e., after natural identifications, rescalings of $\Sigma_{t_i}$ converge in the sense of varifolds to $V$.  
\end{cor}
\begin{proof}
  The existence and finiteness of the limit is an immediate consequence of Theorem \ref{HuiskenMonThm}.  	As the flow has exponential volume growth, it follows from \eqref{KnDecayEstUB} that for any $R>0$ one has
  $$
  \Theta_{\kappa}(t_0,x_0)=\lim_{t\to t_0^-} \int_{\Sigma_t\cap \mathcal{B}_R^g(x_0)} \Phi_{n,\kappa}^{t_0,x_0} (t, x)dVol_{\Sigma_t} (x).
  $$
 It then follows from Lemma \ref{ApproxEuclidHKLem} that $\Theta_{\kappa}(t_0,x_0)$ agrees with the Gaussian density -- here the uniform bound on $\Theta(t)$ on $\Sigma_t$ is a consequence of a straightforward modification of \cite[Proposition 3.9]{Ecker2001} and\cite[Theorem 6.1]{WhiteBoundaryMCF}.   The remaining claims then follow from \cite[Section 6 and 11]{WhiteBoundaryMCF} and  \cite[Theorem 10.2]{KetoverZhou}.

\end{proof}

\section{Colding-Minicozzi $\kappa$-Entropy}
Fix an $(n+k)$-dimensional Cartan-Hadamard manifold $(M,g)$.  Let $\Sigma\subset M$ be a submanifold of dimension $n$.  Recall that, for any $\kappa\geq 0$, the Colding-Minicozzi $\kappa$-entropy of $\Sigma$ in $(M,g)$ is defined to be
$$
\lambda_{g}^\kappa[\Sigma]=\sup_{x_0\in M, \tau>0} \int_{\Sigma} \Phi_{n, \kappa}^{0, x_0}(-\tau, x) dVol_{\Sigma}=\sup_{x_0\in M, \tau>0} \int_{\Sigma} \Phi_{n, \kappa}^{\tau, x_0}(0, x) dVol_{\Sigma}.
$$
It follows from Theorem \ref{HuiskenMonThm} that the $\kappa$-entropy is monotone along well-behaved mean curvature flows in a Cartan-Hadamard manifold with sectional curvature bounded from above by $-\kappa^2$. However, the quantity makes sense independent of its relationship with mean curvature flow and we first record some elementary properties of $\kappa$-entropies that hold on any Cartan-Hadamard manifold.

\subsection{Elementary Properties of  $\kappa$-Entropies}
In what follows we suppose $(M,g)$ is an $(n+k)$-dimensional Cartan-Hadamard manifold.

\begin{prop}\label{VolGrowthProp}
	If $\Sigma\subset M$ is an $n$-dimensional submanifold and  $\lambda_g^{\kappa}[\Sigma]< \infty$ for some $\kappa$, then $\Sigma$ has exponential volume growth.  If $\lambda_g^0[\Sigma]<\infty$, then $\Sigma$ has polynomial volume growth.
	
	In fact, for $n\geq 2, \kappa>0$ and $R\geq \kappa^{-1}$ we have the growth estimate:
	$$
	Vol_{g}(\mathcal{B}_R^g(x_0)\cap \Sigma)\leq C_{n, \kappa}  \lambda_{g}^\kappa[\Sigma] \sqrt{R} e^{\kappa (n-1) R} 
	$$
	and, when $\kappa=0$ or $n=1$ or $R\leq \kappa^{-1}$, 
	$$
	Vol_{g}(\mathcal{B}_R^g(x_0)\cap \Sigma)\leq C_{n, \kappa}  \lambda_g^\kappa[\Sigma] R^n .
	$$
\end{prop}
\begin{proof}
	We first treat the case $\kappa=0.$  In this case, it is immediate from the definition that for $0\leq r \leq R$ and $t=R^2$ one has
	$$
	c_n R^{-n} \leq K_{n,0}(R^2,R)\leq K_{n,0}(R^2,r).
	$$
	Here $1>c_n>0$. 
	Hence, on $\mathcal{B}^g_{R}(x_0)$ one has $\Phi_{n,0}^{0,x_0}(-R^2, x)\geq  c_n R^{-n}$.
	It then follows that
	\begin{align*}
		c_n R^{-n} Vol_g(\mathcal{B}_R^g(x_0)\cap \Sigma)&\leq \int_{\mathcal{B}_R^g(x_0)\cap \Sigma} \Phi_{n,0}^{0,x_0}(-R^2, x) dVol_g(x) \\
		&\leq \int_{\Sigma} \Phi_{n,0}^{0,x_0}(-R^2, x) dVol_g(x) \leq \lambda_g^0[\Sigma].
	\end{align*}
	The claim follows by taking $C_{n,0}=c_n^{-1}$.   We note that $K_{1,\kappa}=K_{1,0}$ and so the same result holds when $n=1$.
	
	We now suppose $\kappa>0$ and $n\geq 2$. 
	As $\partial_r K_{n+1}(t,r)\leq 0$, one has for $0\leq r\leq R$ and $t>0$
	$$
	K_{n,0}(t,R)\leq K_{n,0}(t,r).
	$$
	It follows from \eqref{KnDecayEstLB} that there is a $c'=c'(n)>0$ such that, when $0\leq R\leq 1$ and $t=R^2\leq 1$,
	\begin{align*}
		c' R^{-n-1}  \leq K_{n+1}(R^2, R)\leq K_{n+1}( R^2,r).
	\end{align*}
	Hence, for $\kappa>0$ and $0\leq r \leq  R\leq \kappa^{-1}$, one has
	\begin{align*}
		c' R^{-n-1} & =c' \kappa^{n+1} (\kappa R)^{-n-1}\leq \kappa^{n+1} K_{n+1}(\kappa^2 R^2, \kappa R)= K_{n+1, \kappa}(R^2, R)\\
		&\leq K_{n+1, \kappa}(R^2,r).
	\end{align*}
	Arguing as above gives the claimed bounds when $R\leq \kappa^{-1}$. 
	
	Finally, suppose $R\geq 1$.  For $0\leq r\leq R$ and $t=\frac{1}{n}R$ and some $c''_{n}>0$, one has
	$$
	c'' R^{-\frac{1}{2}} e^{-nR} \leq K_{n+1}\left(\frac{1}{n}R,R
	\right) \leq K_{n+1}\left(\frac{1}{n}R,r\right).
	$$
	Hence, for $\kappa>0$, $R\geq \kappa^{-1}$  and any $0\leq r\leq R$,
\begin{align*}
	c'' \kappa^{n+\frac{1}{2}} R^{-\frac{1}{2}} e^{-\kappa nR}  &\leq  \kappa^{n+1} K_{n+1}\left( \frac{\kappa}{n} R, \kappa R\right) \\
	&= K_{n+1, \kappa}\left(\frac{1}{\kappa n}R, R\right) \leq K_{n+1, \kappa} \left( \frac{1}{\kappa n} R, r\right).
\end{align*}
	Arguing as before, we conclude that
	$$
	c'' \kappa^{n-\frac{1}{2}} R^{-\frac{1}{2}} e^{-\kappa (n-1)R}    Vol_g(\mathcal{B}_R^g(x_0)\cap \Sigma)\leq \lambda_g^{\kappa}[\Sigma]
	$$
	and so the desired estimate holds with $C_{n, \kappa}=\frac{1}{c''}\kappa^{-n+\frac{1}{2}}$.
\end{proof}

\begin{lem}\label{SmallScaleDensityLem}
	If $\Sigma\subset M$ is an $n$-dimensional  properly embedded smooth submanifold such that $\lambda_g^\kappa[\Sigma]<\infty$ for some $\kappa\geq 0$ and $x_0\in M$ is fixed, then 
	$$
	\lim_{\tau\to 0} 	\int_{\Sigma} \Phi_{n, \kappa}^{0,x_0} (-\tau,x) dVol_{\Sigma}(x)\leq 1
	$$
	with equality if and only if $x_0\in \Sigma$. 
\end{lem}
\begin{proof}
	By Proposition \ref{VolGrowthProp}, $\Sigma$ has exponential volume growth.  Hence, by Proposition \ref{ApproxIdentProp} and \eqref{KnDecayEstUB} we have
	$$
	\lim_{\tau\to 0^+} 	\int_{\Sigma} \Phi_{n, \kappa}^{0,x_0} (-\tau,x) dVol_{\Sigma}(x)=\left\{\begin{array}{cc} 1 & x_0 \in \Sigma \\ 0 & x_0 \not\in \Sigma. \end{array} \right. 
	$$ 
\end{proof}

\begin{prop}   If $\Sigma,\Sigma'\subset M$ are $n$-dimensional submanifolds with $\Sigma\subset \Sigma'$ and $0\leq \kappa \leq \kappa_0$, then
	$$
	1\leq \lambda_{g}^{{\kappa}_0}[\Sigma]\leq \lambda_g^\kappa[\Sigma]\leq \lambda_g^\kappa[\Sigma'].
	$$
\end{prop}
\begin{proof}
	By Corollary \ref{KnkappaIneqCor} and the Remark following it, one has, for $t<0$,
	$$
	\Phi_{n, {\kappa}_0}^{t_0, x_0} (t, x)\leq  \Phi_{n, \kappa}^{t_0, x_0}(t,x).
	$$
	Hence,
	$$
	\int_{\Sigma} \Phi_{n, {\kappa}_0}^{t_0, x_0} (t, x)dVol_{\Sigma}\leq \int_{\Sigma}\Phi_{n, \kappa}^{t_0, x_0}(t,x) dVol_{\Sigma}\leq  \int_{\Sigma'}\Phi_{n, \kappa}^{t_0, x_0}(t,x) dVol_{\Sigma'}
	$$
	and the second and third inequalities follow by taking suprema.  The first inequality is an immediate consequence of Lemma \ref{SmallScaleDensityLem}.
\end{proof}
\begin{rem}
	It is possible that $\lambda_g^{{\kappa}_0}[\Sigma]<\infty$  and $\lambda_{g}^{{\kappa}}[\Sigma]=\infty $ for some $0\leq \kappa<{\kappa}_0.$
\end{rem}

\begin{prop}
	If $\Sigma \subset M$ is an $n$-dimensional totally geodesic submanifold on which the induced metric is of constant curvature $\kappa$, then
	$$
	\lambda_{g}^\kappa[\Sigma]=1.
	$$
\end{prop}
\begin{proof}
	By Lemma \ref{SmallScaleDensityLem},  $\lambda_{g}^\kappa [\Sigma]\geq 1$ so we only have to prove the reverse inequality.  To that end, Millison's identity \eqref{MillisonEqn} and the positivity of $K_{n, \kappa}$ imply that  $\Phi_{n,\kappa}^{0,x_0}$ is radially decreasing. 
	Given $x_0\in \Sigma\subset M$ one has by the hypotheses on $\Sigma$ that $\Phi_{n,\kappa}^{0,x_0}(t,x) $ restricts to the backwards heat kernel  and so
	$$
	\int_{\Sigma} \Phi_{n,\kappa}^{0,x_0}(t,x) dVol_{\Sigma}(x)=1.
	$$
	While for $x_0\not \in \Sigma$, let $x_0'\in \Sigma$ be the nearest point in $\Sigma$ to $x_0$.  For $x\in \Sigma$, it follows from the law of cosines and the triangle comparison property of Cartan-Hadamard manifolds that $\dist_{g}(x,x_0')\leq \dist_{g}(x,x_0)$.  Hence,  for all $x\in \Sigma$,
	$$
	\Phi_{n, \kappa}^{0, x_0}(t,x )=K_{n, \kappa}(t, \dist_{g}(x,x_0))\leq K_{n, \kappa}(t, \dist_{g}(x,x_0'))=\Phi_{n, \kappa}^{0, x_0'}(t,x)
	$$
	and the claim follows.
\end{proof}

Finally, we collect observations about the $\kappa$-entropy of closed submanifolds.

\begin{prop}\label{CpctEntProp}
	If $\Sigma\subset M$ is an $n$-dimensional closed submanifold and $\kappa\geq 0$, then for any $\kappa$ such that $\lambda_g^\kappa[\Sigma]>1$, there is a pair $(t_0,x_0)\in  (0, \infty)\times M$, possibly depending on $\kappa$, so
	$$
	\lambda_g^\kappa[\Sigma]=\int_{\Sigma} \Phi_{n, \kappa}^{0, x_0}(-t_0, x) dVol_{\Sigma}(x)=\int_{\Sigma} \Phi_{n,\kappa}^{t_0,x_0}(0,x)dVol_\Sigma(x)<\infty.
	$$
	For such a $\Sigma$ and $\kappa$, if $n\geq 2$ and $ \kappa'>\kappa$, then
	$$	\lambda_g^{\kappa'}[\Sigma]>\lambda_g^{\kappa}[\Sigma]>1.
	$$
	For such a $\Sigma$ and $\kappa$, if $\Sigma'\subset M$ is an $n$-dimensional submanifold with $\Sigma\subset \Sigma'$ and $\Sigma'\setminus \Sigma \neq \emptyset$, then
	$$
\lambda_g^{\kappa}[\Sigma']>\lambda_g^\kappa[\Sigma]>1.
	$$
\end{prop}
\begin{proof}
	Fix a point $x_1\in M$ and $\kappa\geq 0$.  As $\Sigma$ is closed, there is an $R_1$ so $\Sigma\subset \mathcal B_{R_1}^{g}(x_1)$.  
	By the decay estimates on $K_n$ of \eqref{KnDecayEstUB}, for any $\epsilon>0$ there is an $R_{\epsilon, \kappa}>R_0$, such that if $x_0\in  M\setminus  \mathcal B_{R_{\epsilon, \kappa}}^{g}(x_1)$, then, for all $\tau>0$, one has
	$$
	\int_{\Sigma} \Phi_{n, \kappa}^{0,x_0} (-\tau,x) dVol_{\Sigma}(x)< \epsilon.
	$$
	Similarly, by \eqref{TimeDecayEqn} for $x_0\in  \mathcal B_{R_1}^{g}(x_1)$ one has
	$$
	\lim_{\tau\to \infty} 	\int_{\Sigma} \Phi_{n, \kappa}^{0,x_0} (-\tau,x) dVol_{\Sigma}(x)=0.
	$$
	As $\Sigma$ is smooth and automatically has exponential volume growth, by Lemma \ref{SmallScaleDensityLem}, for any fixed $x_0$, one has
	$$
	\lim_{\tau\to 0} 	\int_{\Sigma} \Phi_{n, \kappa}^{0,x_0} (-\tau,x) dVol_{\Sigma}(x)\leq 1.
	$$
	Hence, if $\lambda_g^\kappa[\Sigma]>1$ there must be a $\tau_1\in (0,1)$ and an $R_2>R_1$ such that
	$$
	\lambda_{\kappa}^g[\Sigma]=\sup_{x_0\in  \bar{\mathcal{B}}_{R_2}^{g}(x_1), \tau\in [\tau_1, \tau_1^{-1}]} \int_{\Sigma} \Phi_{n,\kappa}^{0,x_0} (-\tau,x) dVol_{\Sigma}(x).
	$$
	This is the supremum of a continuous function over a compact set, so it is achieved and has a finite value.  This proves the first claim.  
	
	For the second claim observe that, as $n\geq 2$, we have, by Corollary \ref{KnkappaIneqCor}, that 
	$$\Phi_{n, {\kappa}_0}^{t_0, x_0} (t, x)<  \Phi_{n, \kappa}^{t_0, x_0}(t,x).$$
	As we may apply the first claim to $\Sigma$ and $\kappa$,  there is a $(t_0,x_0)$ so
	$$
	\lambda_g^{\kappa}[\Sigma]=\int_{\Sigma} \Phi_{n, {\kappa}}^{t_0, x_0} (0, x)dVol_{\Sigma}(x)<\int_{\Sigma} \Phi_{n, {\kappa}'}^{t_0, x_0} (0, x)dVol_{\Sigma}(x)\leq \lambda_g^{\kappa'}[\Sigma].
	$$

	As $\Sigma\subset \Sigma'$, $\Sigma\setminus \Sigma'\neq \emptyset$ and $\Phi_{n, \kappa}^{t_0, x_0}(0, \cdot)$ is everywhere positive.
	$$
	\lambda_g^\kappa [\Sigma]=\int_{\Sigma} \Phi_{n,\kappa}^{t_00, x_0} (0, x) dVol_{\Sigma}< \int_{\Sigma'}\Phi_{n,\kappa}^{t_0, x_0} (0, x) dVol_{\Sigma'}\leq \lambda_g^\kappa[\Sigma'].
	$$
	Again the $(t_0,x_0)$ exist based on the hypotheses on $\Sigma$ and $\kappa$.
\end{proof}

\subsection{Monotonicity of $\kappa$-entropies}
We next use the monotonicity formula to show Theorem \ref{MainMonThm}, i.e., that the $\kappa$-entropy is monotone along a well-behaved mean curvature flow inside a Cartan-Hadamard manifold with sectional curvature bounded from above by $-\kappa^2_0$.
\begin{proof}[Proof of Theorem \ref{MainMonThm}]
	Given $\epsilon>0$, the definition of $\lambda_g^\kappa$ ensures there is a $\tau>0$ and $x_0\in M$ so
	$$
	\lambda_g^\kappa[\Sigma_{t_2}]-\epsilon\leq \int_{\Sigma_{t_2}} \Phi_{n,\kappa}^{0,x_0}(-\tau, x)d\mathrm{Vol}_{\Sigma_{t_2}}= \int_{\Sigma_{t_2}} \Phi_{n,\kappa}^{\tau+t_2,x_0}(t_2, x)d\mathrm{Vol}_{\Sigma_{t_2}}.
	$$
	The second equality follows from the definition of the $\Phi_{n,\kappa}^{t_0,x_0}$. 
	The curvature hypotheses and volume bound imply Theorem \ref{HuiskenMonThm} holds and so
	\begin{align*}
		\lambda_g^\kappa[\Sigma_{t_2}]-\epsilon&\leq \int_{\Sigma_{t_2}} \Phi_{n,\kappa}^{\tau+t_2,x_0}(t_2, x)d\mathrm{Vol}_{\Sigma_{t_2}}\leq \int_{\Sigma_{t_1}} \Phi_{n,\kappa}^{\tau+t_2,x_0}(t_1, x)d\mathrm{Vol}_{\Sigma_{t_1}}\\
		&=\int_{\Sigma_{t_1}} \Phi_{n,\kappa}^{0,x_0}(-\tau+t_1-t_2, x)d\mathrm{Vol}_{\Sigma_{t_1}}\leq \lambda_g^\kappa[\Sigma_{t_1}].
	\end{align*}
The first conclusion follows as $\epsilon>0$ was arbitrary.	

To verify the last claim, observe that when the flow is a classical mean curvature flow of a closed manifold, $F:[t_1,t_2]\times \Sigma\to M$, then $F([t_1,t_2]\times \Sigma)$ is a compact subset of $M$ and so there is automatically a uniform lower bound, $\Lambda$, on $Ric_g$ along the flow.  In addition, in a Cartan-Hadamard manifold, the sectional curvatures are, by definition, non-positive and so one may always take $\kappa=\kappa_0=0$.

\end{proof}
Because the  result becomes technically simpler for classical flows of closed submanifolds it yields information about the $\kappa$-entropy of closed hypersurfaces.
\begin{prop}\label{LowerBoundProp}
		Suppose $(M,g)$ is an $(n+1)$-dimensional Cartan-Hadamard manifold with $\sec_{g}\leq -{\kappa}^2_0$, $0\leq \kappa\leq \kappa_0$ and $\Sigma\subset M$ is a closed hypersurface.  There is a constant $\epsilon(n)>0$ so
		$$
		\lambda_g^\kappa[\Sigma]\geq 1+\epsilon(n)>1.
		$$
		Moreover if either $\Sigma$ is mean convex or $n=1$ or $2$, then,
		$$
	\lambda_g^\kappa[\Sigma]\geq \lambda[\mathbb{S}^n]>1.
		$$
		When $\kappa>0$, the inequality is strict.
		Furthermore, if $\Sigma$ is connected but not diffeomorphic to $\mathbb{S}^n$, then
			$$
		\lambda_g^\kappa[\Sigma]> \lambda[\mathbb{S}^{n-1}]>\lambda[\mathbb{S}^n].
		$$
\end{prop}
\begin{proof}
	 Fix a point $x_1\in M$ and choose $R_1\geq 2$ large enough so $\Sigma \subset \mathcal{B}^g_{\frac{1}{2}R_1}(x_1)$.   As $(M,g)$ is a Cartan-Hadamard manifold it follows from the mean curvature comparison theorem that $\partial \mathcal{B}^g_{R_1}(x_1)$ is smooth and strictly convex -- see Proposition \ref{HessCompProp} or \cite[Section 4.8]{Jost2008}.   In particular, one has short time existence of the classical mean curvature flow, $t\in [0, T)\mapsto \Sigma_t$ of $\Sigma_0=\Sigma$ and this flow remains inside of $\mathcal{B}^g_{R_1}(x_1)$.  In fact, we may take $T\in (0, \infty]$ to be the maximal smooth time of existence of the classical flow.  As there are no closed minimal surfaces in a Cartan-Hadamard manifold, $T<\infty$ and the flow must also form a singularity in finite time.  If $\Sigma_0$ is mean convex, then it follows from standard computations on the evolution of mean curvature and the parabolic maximum principle that $\Sigma_t$ remains mean convex.   Note that there is a uniform lower bound on $Ric_g$ in $\mathcal{B}^g_{R_1}(x_1)$ and so Theorem \ref{HuiskenMonThm} applies to the flow automatically.
	
	At the first singular time, $T>0$, there is at least one singular point $x_0\in \mathcal{B}_{R_1}^g(x_1)$.  One may take a parabolic blow-up around $(x_0, T)$, and it follows from Theorem \ref{HuiskenMonThm} and White's version of Brakke regularity \cite{WhiteReg}
	that,  up to passing to a subsequence, the flow converges to a non-flat self-shrinker in $\Real^{n+1}$.  As this shrinker is the limit of classical flows, \cite{WhiteReg} applies and we obtain an $\epsilon=\epsilon(n)>0$ so
	$$
	\lambda_g^\kappa[\Sigma]\geq \Theta_\kappa(T,x_0)\geq 1+\epsilon(n)>1.
	$$
	This proves the first claim.
	
	 When the flow is mean convex it follows from \cite{HASLHOFER20181137} -- see also \cite{ HuiskenSinestrariConvexity,WhiteMC1, WhiteMC2} -- that the singularity must be a non-flat generalized cylinder and hence has entropy at least $\lambda[\mathbb{S}^n]$.  Likewise, when  $n=1$ the classification of embedded shrinkers in \cite{AbreschLanger} and when $n=2$ the classification of low entropy shrinkers in \cite{BernsteinWang2} imply the entropy is at most $\lambda[\mathbb{S}^n]$ -- here we use that an immersed shrinker has entropy at least $2$.  The strictness of the inequality when $\kappa>0$ is an immediate consequence of the strictness of the monotonicity in Theorem \ref{HuiskenMonThm} when $\kappa>0$ and the fact that the flow consists of closed hypersurfaces -- which cannot be geodesic cones.
    
    To conclude the proof observe that in either case above if $	\lambda_g^\kappa[\Sigma]\leq \lambda[\mathbb{S}^{n-1}]$, then the shrinker obtained must be the round sphere $\sqrt{2n} \mathbb{S}^n$ and the convergence of the rescaled flow must be locally smooth and converge to the round sphere.  As $\Sigma$ is connected this means $\Sigma$ is diffeomorphic to the sphere.
 	 
\end{proof}
\section{Rigidity of some $\kappa$-entropy minimizers}

%

Observe that for $\kappa_0\geq \kappa\geq 0$ the lower bound on entropy for closed hypersurfaces in an $(n+1)$-dimensional Cartan-Hadamard manifold, $(M,g)$, with curvature bounded from above by $-\kappa^2_0$ is sharp in the sense that, for any $x\in M$,
$$
\lim_{R\to 0} \lambda_g^\kappa[\partial \mathcal{B}_R^g(x)]=\lambda[\mathbb{S}^n].
$$
However, for $\kappa>0$, by Proposition \ref{LowerBoundProp}, equality is never achieved. Indeed,
$$
\lambda_g^\kappa[\Sigma]>\lambda[\mathbb{S}^{n}]
$$
where $\Sigma$ is a hypersurface for which Proposition \ref{LowerBoundProp} applies.  When $\kappa=0$ it is possible for the lower bound to be achieved, but this implies rigidity of both the hypersurface $\Sigma$ and the metric $g$ in the region $\Sigma$ encloses. 
\begin{thm}\label{RigidityThm}
	Suppose $(M,g)$ is an $(n+1)$-dimensional Cartan-Hadamard manifold and that $\Omega\subset M$ is a compact domain with $\Sigma=\partial \Omega$ smooth and mean convex.  If
	$$
	\lambda_g^0[\partial \Omega]= \lambda[\mathbb{S}^n],
	$$
	then there is an $R>0$ so $(\Omega, g|_{\Omega})$ is isometric to $(\bar{B}_R, g_{\Real}|_{\bar{B}_R})$, the standard Euclidean metric restricted to the closed Euclidean ball in $\Real^{n+1}$ of radius $R$.
	When $n=1$ or $2$, the same result is true without the hypothesis that $\partial \Omega$ is mean convex.
\end{thm}

%

To prove this result we first use a result of Huisken \cite{Huisken1986} to verify it for sufficiently convex sets.
\begin{prop}\label{RigidConvexProp}
Suppose $(M,g)$ is an $(n+1)$-dimensional Cartan-Hadamard manifold. Fix a compact set $V\subset M$ and let
$$
\gamma = n \max_{x\in V, P\in G_2(T_xM)} |\mathrm{sec}_{g} (P)|+n^2 \max_{x\in V} |\nabla_g \mathrm{Riem}_g|.
$$
If $\Omega\subset V$ is a compact, geodesically convex domain with $\partial \Omega$ smooth, connected, and with all principal curvatures with respect to the inward normal strictly larger than $\gamma$,
then there is a point $x_0\in \mathrm{int}(\Omega)$, a time $T\in (0, \infty)$, and a unique classical mean curvature flow $t\in [0, T)\mapsto \Sigma_t$ that disappears at $x_0$ at time $T$ satisfying $\Sigma_0=\partial \Omega$, and for all $t\in [0, T)$, $\Sigma_t=\partial \Omega_t$ where $\Omega_t$ are compact and geodesically convex. 

If, in addition,  $\Omega$ satisfies
	$$
\lambda_g^0[\partial \Omega]= \lambda[\mathbb{S}^n],
$$
	then there is an $R>0$ so $(\Omega, g|_{\Omega})$ is isometric to $(\bar{B}_R, g_{\Real}|_{\bar{B}_R})$ the standard Euclidean metric restricted to a closed Euclidean ball in $\Real^{n+1}$ of radius $R$.
\end{prop}
\begin{proof}
	We first observe that as $(M,g)$ is a Cartan-Hadamard manifold, for any fixed $x_1\in M$ and any $r>0$,  $\bar{\mathcal{B}}^g_r(x_1)$, the closed ball around $x_1$ in $M$  with respect to the metric $g$ of radius $r$ is geodesically convex and has $\partial \bar{\mathcal{B}}^g_r(x_1)$ smooth with positive second fundamental form -- see \cite[Section 4.8]{Jost2008}.

	Fix a point $x_1\in M$ and choose $R_1>0$ large enough such that $V\subset \mathcal{B}^g_{R_1}(x_1)$. 
	We immediately observe that the choice of $\gamma$ implies that the mean curvature of $\partial \Omega$ with respect to the inward normal satisfies $H_{\partial \Omega}> n\gamma \geq 0$.
	The choice of $\gamma$ and the fact that $(M,g)$ is a Cartan-Hadamard manifold  ensures that the second fundamental form of $\partial \Omega$ satisfies the hypotheses of \cite[Theorem 1.1]{Huisken1986}.  As $\partial \Omega \subset \bar{\mathcal{B}}^g_{R_1}(x_1)$ and $\partial \bar{\mathcal{B}}^g_{R_1}(x_1)$ has positive mean curvature with respect to the inward normal, we may apply \cite[Corollary 1.4]{Huisken1986} to $\Sigma_0=\partial \Omega$ in order to obtain a unique classical mean curvature flow $t\in [0,T)\mapsto \Sigma_t$.
	The result of \cite{Huisken1986} ensures that each $\Sigma_t$ has strictly positive principle curvatures with respect to the inward normal and that it becomes singular at time $T$ and disappears in a round point at $x_0$.  Clearly, the flow moves strictly inward and so $x_0\in \mathrm{int}(\Omega)$. This proves the first claim.
	
	To conclude the proof, observe that as the mean curvature flow disappears in a round point at $x_0$,  Corollary \ref{DensityLem} ensures
	$$
	\lim_{t\to T} \int_{\Sigma_t} \Phi_{n, 0}^{T, x_0} (t, y)dVol_{\Sigma_t} (y) = \lambda[\mathbb{S}^n].
	$$
    The entropy assumption and Theorem \ref{HuiskenMonThm} imply that for all $t\in [0, T)$ one has
    \begin{equation}\label{MonConstantEqn}
   \int_{\Sigma_t} \Phi_{n, 0}^{T, x_0} (t, y)dVol_{\Sigma_t} (y) = \lambda[\mathbb{S}^n]= \lambda_g^0[\Sigma_t].
\end{equation}
    Now choose compact sets $\Omega_t$ so $\Omega_t\subset \Omega$ and $\partial \Omega_t=\Sigma_t$.  One verifies that, for $0\leq t_0<t_1<T$, 
   $$
    x_0\in \Omega_{t_1}\subset \mathrm{int}(\Omega_{t_0})\subset \Omega.
    $$
  Moreover, as $(M,g)$ is a Cartan-Hadamard manifold and $\Sigma_t=\partial \Omega_t$ has positive second fundamental form with respect to the inward normal, $\Omega_t$ is geodesically convex. 
  In particular, for each $t\in (0, T)$ and each $y\in \partial \Omega_t$, the unique unit speed geodesic connecting $x_0$ to $y$, $\gamma_{y}:[0,L]\to M$ satisfies $\gamma_y(s)\in \Omega_t$ for all $s\in [0, L]$ and $\gamma_y'(L)\in T_y M$ is transverse to $T_y \Sigma_t$.  It follows from Proposition \ref{MonotonicityFormulaProp} and Lemma \ref{QPosLem} and \eqref{MonConstantEqn} that for all $s\in (0, L]$ one has
  $$
  \sec_g(P)=0
  $$
  where $P$ is any plane in $T_{\gamma_y(s)}M$ containing $\gamma'_y(s)$.  In particular, if $\rho$ is the distance function to $x_0$, then one has
  $$
   \mathrm{sec}_{g}(\nabla_g\rho \wedge W)=0
   $$
   for all $y\neq x_0$ and $W$ orthogonal to $\nabla_g \rho(y)$.  Arguing as Lemma \ref{SupersolutionLem} this means the sectional curvature of $g$ vanishes in $\Omega$.
   
   In particular, there is an isometric immersion $F: (\Omega, g|_{\Omega})\to (\Real^{n+1}, g_{\Real})$.  This map is actually an embedding as $\Omega$ is geodesically convex. Moreover, one has that
    $$\lambda[F(\partial \Omega)]=\lambda_g^0[\partial \Omega]=\lambda[\mathbb{S}^n]$$
     and so, by \cite[Theorem 1.1]{BernsteinWang1} and \cite[Theorem 0.1]{JZhu}, up to translation, $F(\Omega)=\bar{B}_{R}(0)$, a closed Euclidean ball of radius $R$ for some $R>0$.
\end{proof}

\begin{proof}[Proof of Theorem \ref{RigidityThm}]
  We first suppose that $\partial \Omega $ is connected.
	As $\Omega $ is compact, for any fixed $x_1$, there is an $R_1>0$ sufficiently large so $\Omega\subset \mathcal{B}_{\frac{1}{2}R_1}^g(x_1)$.  As $(M,g)$ is a Cartan-Hadamard manifold one has that $\partial \mathcal B_{R_1}^g(x_1)$ is smooth and strictly mean convex.   In particular, the classical mean curvature flow, $t\in [0,T)\mapsto \Sigma_t$ of $\Sigma_0=\partial \Omega$ exists uniquely and remains inside of $\mathcal B_{R_1}^g(x_1)$.  As there are no closed minimal surfaces in a Cartan-Hadamard manifold the flow must also form a singularity in finite time -- we take $T<\infty$ to be the maximal smooth time of existence.  As shown in the proof of Proposition \ref{LowerBoundProp} the hypotheses on $\Omega$ ensure there is an $x_0\in {\mathcal{B}}^g_{R_1}(x_1)$ so the flow forms a spherical singularity at time $T$ and at the point $x_0$.


  With $V=\bar{\mathcal{B}}^g_{R_1}(x_1)$, pick $\gamma>0$ as in Proposition \ref{RigidConvexProp}.   We may choose compact sets $\Omega_t$ so $\partial \Omega_{t}=\Sigma_t$ and $\Omega_t\subset V$.   As $x_0\in {\mathcal{B}}^g_{R_1}(x_1)$ and the flow forms a spherical singularity at $(T,x_0)$, there is an $\epsilon>0$ small that ensures $\Sigma_{T-\epsilon}$ has principal curvatures greater than $\gamma$ with respect to the inward normal to $\Omega_{T-\epsilon}$.  As $\lambda_g^0[\Sigma_{T-\epsilon}]=\lambda[\mathbb{S}^n]$ it follows from Proposition \ref{RigidConvexProp} that
  	$(\Omega_{T-\epsilon}, g|_{\Omega_{T-\epsilon}})$ is isometric to $(\bar{B}_{R(T-\epsilon)}, g_{\Real}|_{\bar{B}_{R(T-\epsilon)}})$.  Moreover,  for all $t\in [T_\epsilon, T)$, it is true that $(\Omega_t, g|_{\Omega_t})$ is isometric  to $(\bar{B}_{R(t)}, g_{\Real}|_{\bar{B}_{R(t)}})$ where one easily checks that $R(t)=\sqrt{2n(T-t)}$. Define a set of good times by  	
  	$$
  \mathcal{G}=\set{t\in [0, T): (\Omega_t, g|_{\Omega_t}) \mbox{ isometric to }(\bar{B}_{R(t)}, g_{\Real}|_{B_{R(t)}})}.
  	$$
  	Clearly, $[T-\epsilon, T)\subset \mathcal{G}$ and $\mathcal{G}$ is connected.  Set $t_0=\inf \mathcal{G}$.  If $t_0=0$, then we have verified the result when $\partial \Omega$ is connected. If $t_0>0$ then,  we can set 
  	$$
  	\gamma_0=(2n(T-t_0))^{-1/2}>0
  	$$
  	which is the value of the principle curvatures of $\Sigma_{t_0}$. 
  	As $\Omega_{t_0}$ is isometric to flat space, for $t'<t_0$ sufficiently close to $t_0$ we can ensure that the constant $\gamma$ given by Proposition \ref{RigidConvexProp} for the region $\Omega_{t'}$ is less than $\frac{1}{2}\gamma_0$.
  Hence, Proposition \ref{RigidConvexProp} also applies to $\Omega_{t'}$ which immediately implies that $t'\in \mathcal{G}$, a contradiction.
  	
   Finally, we suppose $\partial \Omega$ has  more than one component.  Let $\Sigma'$ be a choice of component and suppose $\Omega'$ is the compact region bounded by $\Sigma'$ (such a region exists by Alexander duality).  As $\Sigma'$ is compact and, by Proposition \ref{LowerBoundProp} satisfies $\lambda[\Sigma']>1$, it follows from Proposition \ref{CpctEntProp} that $\lambda[\Sigma']< \lambda[\Sigma_0]= \lambda[\mathbb{S}^n]$, but applying the previous argument to $\Omega'$ shows that $\lambda[\Sigma']=\lambda[\mathbb{S}^n]$ which is a contradiction and so $\partial \Omega$ has only one component. 
\end{proof}

%
%
\bibliographystyle{hamsabbrv}
\bibliography{Library2}
\end{document}